\documentclass[a4paper,11pt]{amsart}

\usepackage{amsmath}
\usepackage{amssymb}
\usepackage{amsthm}
\usepackage{verbatim}
\usepackage[all]{xy}
\usepackage{enumerate}

\usepackage{color}

\usepackage[all]{xy}

\usepackage{graphicx}

\makeatletter
\@namedef{subjclassname@2010}{%
	\textup{2010} Mathematics Subject Classification}
\makeatother

\newtheorem{thm}{Theorem}[section]
\newtheorem{cor}[thm]{Corollary}
\newtheorem{lem}[thm]{Lemma}
\newtheorem{prop}[thm]{Proposition}

\theoremstyle{definition}
\newtheorem{dfn}[thm]{Definition}
\newtheorem{rmk}[thm]{Remark}

\numberwithin{equation}{section}

\numberwithin{equation}{section}
\newcommand{\cK}{\mathcal{K}}

\newcommand{\cM}{\mathcal{M}}
\newcommand{\cH}{\mathcal{H}}

\newcommand{\id}{\textrm{id}}
\newcommand{\sG}{\mathsf{G}}

\newcommand{\Id}{\textrm{Id}}
\newcommand{\Dom}{\textrm{Dom}}

\newcommand{\cB}{\mathcal{B}}

\newcommand{\cHR}{\mathcal{H}_{\mathbb{R}}}
\newcommand{\cHC}{\mathcal{H}_{\mathbb{C}}}

\newcommand{\cF}{\mathcal{F}}
\newcommand{\sign}{{\rm sign}}
\newcommand{\cR}{\mathcal{R}}
\newcommand{\cE}{\mathcal{E}}

\newcommand{\BMO}{{\rm bmo}}
\newcommand{\BMOa}{{\rm BMO}}
\newcommand{\sfG}{\mathsf{G}}
\newcommand{\cN}{\mathcal{N}}

\newcommand{\cO}{\mathcal{O}}
\newcommand{\dom}{{\rm Dom}}
\newcommand{\cS}{\mathcal{S}}

\newcommand{\cU}{\mathcal{U}}

\newcommand{\cP}{\mathcal{P}}

\begin{document}




\title[Harmonic analysis and BMO-spaces of free Araki-Woods factors]{Harmonic analysis and BMO-spaces of  free Araki-Woods factors}

\author[M. Caspers]{Martijn Caspers}
\address{TU Delft, EWI/DIAM, 
	P.O.Box 5031, 
	2600 GA Delft, 
	The Netherlands}
\email{m.p.t.caspers@tudelft.nl}
 
\date{\today}

\begin{abstract}
We consider semi-group BMO-spaces associated with arbitrary von Neumann algebras and prove interpolation theorems. This extends results by Junge-Mei for the tracial case. We give examples of multipliers on free Araki-Woods algebras and in particular we find  $L_\infty \rightarrow {\rm BMO}$ multipliers.  We also provide $L_p$-bounds for a natural generalization of the Hilbert transform.  
\end{abstract}



\maketitle

\section{Introduction}
Recall that the BMO-norm of a classical integrable function $f: \mathbb{R}^n \rightarrow \mathbb{C}$ is defined as 
\[
\Vert f \Vert_{\BMOa}  =   \sup_{Q \in \mathcal{Q}} \frac{1}{\vert Q \vert} \int_Q   \vert f(s) - \oint_Q f \vert^2  ds,
\]
where $\oint_Q f$ is the average of $f$ over $Q$ and $\mathcal{Q}$ is the set of all cubes in $\mathbb{R}^n$. The importance of the BMO-norm and BMO-spaces lies in the fact that they arise as end-point estimates/spaces for the bounds of linear maps on function spaces on $\mathbb{R}^n$. This includes many singular integral operators, Calder\'on-Zygmund operators and Fourier multipliers. BMO-spaces are by Fefferman-Stein duality \cite{FeffermanStein}    dual  to Hardy spaces and provide optimal bounds for the Hilbert transform. By interpolation BMO-spaces form an effective tool to obtain $L_p$-bounds of multipliers. 

BMO-spaces can also be studied through semi-groups.  Consider for example the heat semi-group $\mathcal{S} := (\Phi_t)_{t \geq 0} := (e^{-t \Delta})_{t \geq 0}$ with Laplacian $\Delta$ acting on $L_\infty(\mathbb{R}^n)$. Then alternatively the BMO-norm may be realized through an equivalent (semi-)norm
\[
\Vert f \Vert_{\BMO_{\mathcal{S}}} = \sup_{t \geq 0}  \Vert \vert \Phi_t(f) \vert^2 - \Phi_t(\vert f \vert^2) \Vert^{\frac{1}{2}}. 
\]
  BMO-spaces associated with more general semi-groups  were first studied in \cite{StroockVar}, \cite{Varopoulos} and much more recently in  \cite{XuanYan1}, \cite{XuanYan2}.   
   See also \cite{Grafakos1}, \cite{Grafakos2}. These concern semi-groups on measure spaces, which from our viewpoint is the commutative situation. 
   
   \vspace{0.3cm}
   
   The development and exploration of structural properties of C$^\ast$-algebras and von Neumannn algebras led to the demand of a thorough development of harmonic analysis on non-commutative spaces. After the founding work by Eymard defining the Fourier algebra of a group \cite{Eymard}, the study of its $L_\infty$-multipliers turned out to have tremendous impact on the structure of operator algebras (see e.g. \cite{NateTaka}).    
      In recent years also  the $L_p$-theory was pursued. Under suitable H\"ormander-Mikhlin type conditions several multiplier theorems were established  for group von Neumann algebras  \cite{JMP}, \cite{CPPR},  \cite{GJP1} and vector valued harmonic analysis \cite{Cadilhac}, \cite{ParcetJFA}. On quantum spaces several surprising multiplier theorems have been achieved \cite{XuCMP}, \cite{RicardJFA}. See also \cite{Xiong}, \cite{GJP2}.    
    These results naturally raise questions about end-point estimates and optimal bounds for multipliers. 
    
    Parallel to this development semi-groups on non-commutative measure spaces have played a more and more important role in recent years. They lead to strong applications in non-commutative potential theory  and quantum probability, see e.g.  \cite{CiprianiSauvageot},   \cite{CFK}. Semi-groups naturally appear in approximation properties of von Neumann algebras \cite{JolissaintMartin}, \cite{CaspersSkalskiII}. Also the approach by Ozawa-Popa \cite{OzawaPopa} and Peterson \cite{Peterson} yields new deformation-rigidity properties of von Neumann algebras through the theory of semi-groups and derivations (see also \cite{Avsec}). 
    
    In \cite{JungeMei} Junge and Mei pursued the theory of non-commutative semi-group BMO-spaces associated with non-commutative measure spaces. They introduce several notions of BMO starting from a Markov semi-group on a tracial von Neumann algebra. Relations between these spaces are studied and interpolation results are obtained. A crucial ingredient of their approach is formed by Markov dilations of semi-groups that allows one to `intertwine' semi-group BMO-spaces with BMO-spaces associated with martingales and derive results from this probabilistic martingale setting. 
    
    The first aim of this paper is the study of BMO-spaces associated with an arbitrary $\sigma$-finite von Neumann algebra. We take the natural definition using  a faithful normal state which is not necessarily tracial anymore as a starting point. We extend  interpolation results from \cite[Theorem 5.2]{JungeMei} to the arbitrary setting under a modularity assumption on the Markov semi-group. The modularity assumption is necessary to carry out our proof through Haagerup's reduction method and due to the fact that the probabilistic martingale BMO-spaces in \cite{JungePerrin} are studied (in principle only) in the tracial setting. This culminates in Theorem \ref{Thm=InterpolationBMO}, which briefly states the following.
      Let $\cS$ be a modular Markov semi-group admitting a reversed Markov dilation with a.u. continuous path on a $\sigma$-finite von Neumann algebra $\cM$.  We have  
    \begin{equation}\label{Eqn=ThmA}
    [\BMO^\circ_{\mathcal{S}}(\cM), L_p^\circ(\cM)]_{1/q} \approx_{pq} L_{pq}^\circ(\cM).
    \end{equation}
         Other interpolation theorems for Poisson semi-groups and different BMO-spaces are then discussed in Section \ref{Sect=OtherBMO}. Proofs here are similar  and some aspects in fact simplify. 
     
     In Section \ref{Sect=Hilbert} we give examples of multiplier theorems of non-tracial von Neumann algebras, namely free Araki-Woods factors (see \cite{Shlyakhtenko}). The first part of Section \ref{Sect=Hilbert} introduces a natural generalization of the (free) Hilbert transform. We get $L_p$-bounds through Cotlar's trick. Recently in  \cite{MeiRicard} Mei and Ricard obtained the analogous result for free group factors. We also give examples of $L_\infty \rightarrow {\rm BMO}$ multipliers  and show that the interpolation result of \eqref{Eqn=ThmA} applies. We leave it as an open question whether the Hilbert transform admits a $L_\infty \rightarrow {\rm BMO}$-estimate (or even a ${\rm BMO} \rightarrow {\rm BMO}$-estimate as for the classical Hilbert transform \cite{FeffermanStein}, \cite{Grafakos2}). In Section \ref{Sect=Dilation} we construct a reversed Markov dilation for the semi-groups that we use on free Araki-Woods factors. The construction is essentially due to Ricard \cite{RicardDilation} which is combined with an ultraproduct argument to go from the discrete to continuous case.

    \vspace{0.3cm}

\section{Preliminaries and notation}
We start with some general conventions. For general operator theory we refer to \cite{TakesakiI} and for operator spaces to \cite{EffrosRuan}, \cite{Pisier}.
Throughout the paper $\cM$ will be a von Neumann algebra with fixed normal faithful state $\varphi$. $\mathcal{S} = (\Phi_t)_{t \geq 0}$ will be a fixed Markov semi-group, see Section \ref{Sect=Semigroup} for details. $(\sigma_s^{\varphi})_{s \in \mathbb{R}}$ denotes the modular automorphism group of $\varphi$, see \cite{TakII} for modular theory.

\subsection{General notation} 
For the complex interpolation method we refer to the book  \cite{BerghLofstrom}. See also   \cite{CaspersLpf} for a short summary   and the relation to non-commutative $L_p$-spaces. Let $S$ be the strip of all complex numbers with imaginary part in the interval $[0,1]$. For a compatible couple of Banach spaces $(X,Y)$  denote $\mathcal{F}(X,Y)$ for the space of functions $S \rightarrow X+Y$ that (i) are continuous on $S$  and analytic on the interior of $S$, (ii) $f(s) \in X$ and $f(i + s) \in Y$, (iii) $\Vert f(s) \Vert_X \rightarrow 0$ and  $\Vert f(i+s) \Vert_Y \rightarrow 0$  as $\vert s \vert \rightarrow \infty$. We write $(X,Y)_\theta$ for the interpolation space at parameter $\theta \in [0,1]$.

\subsection{$L_p$-spaces associated with an arbitrary von Neumann algebra}\label{Sect=Lp}
This  paper establishes results on interpolation and harmonic analysis on non-tracial von Neumann algebras. The $L_p$-spaces of such von Neumann algebras can be described through constructions introduced by Haagerup \cite{HaagerupLp}, \cite{TerpI} and Connes-Hilsum \cite{Connes}, \cite{Hilsum} (the latter in fact relies on Haagerup's construction to treat sums and products of unbounded operators). In principle we use the definition of Hilsum \cite{Hilsum}, though it is easy to recast each of the statements in terms of \cite{HaagerupLp}. 

For a general von Neumann algebra $\cM$ we let $\phi'$ be a fixed normal, semi-finite, faithful  weight on the commutant $\cM'$. For a normal, semi-finite weight $\varphi$ on $\cM$ we write $D_\varphi$ for Connes's spatial derivative $d\varphi/d\phi'$  \cite{Connes}, \cite{TerpI}. For every von Neumann algebra in this paper $\phi'$ is implicitly fixed; it can be chosen arbitrary and  $\phi'$ will be suppressed in the notation. 
 $L_p(\cM)$ with $\cM \subseteq B(\cH)$ is defined as all closed densely defined operators $x$ on $\cH$ such that $\vert x \vert^p = D_{\varphi}$ for some $\varphi \in \cM_\ast^+$. Then $\Vert x \Vert_p = \Vert \varphi \Vert^{1/p}$. Products and sums of elements in (different) $L_p$-spaces are understood as strong products and strong sums (so closure of the product and sum). We will omit these closures in the notation. $L_p$-spaces satisfy classical properties as H\"older estimates.  In particular for all $x \in \cM$ and $\varphi \in \cM_\ast$ positive we have  $D_{\varphi}^{\frac{1}{2p}} x D_{\varphi}^{\frac{1}{2p}} \in L_p(\cM)$. In fact such elements are (norm) dense in $L_p(\cM)$ for $1 \leq p < \infty$. 

We turn $L_p(\cM), 1 \leq p \leq \infty$  into a compatible couple (or compatible scale) of Banach spaces. Assume $\cM$ is $\sigma$-finite, meaning that there exists a faithful, normal state $\varphi$ on $\cM$. Then there is a contractive embedding $\kappa_p^{\varphi}: L_{p}(\cM) \rightarrow L_{1}(\cM)$ determined by 
\[
D_{\varphi}^{\frac{1}{2p}} x D_{\varphi}^{\frac{1}{2p}} \mapsto D_{\varphi}^{\frac{1}{2}} x D_{\varphi}^{\frac{1}{2}}.
\]
Considering $L_p(\cM)$  as (non-isometric) linear subspaces of $L_1(\cM)$ we may and will interpret intersections, sum spaces and interpolation spaces of $L_p(\cM)$  and $L_r(\cM)$ within $L_1(\cM)$. Such spaces depend on $\varphi$ and we will usually mark $\varphi$ in the notation (we shall need a transition between the tracial and non-tracial case). For example $[L_p(\cM), L_r(\cM)]_\theta^{\varphi}$ will denote the complex interpolation spaces between $L_p(\cM)$ and $L_r(\cM)$ at parameter $\theta \in [0,1]$ with respect to the embeddings of $L_p(\cM)$ and $L_r(\cM)$ in $L_1(\cM)$ through $\kappa_p^{\varphi}$ and $\kappa_r^{\varphi}$.




\subsection{Semi-groups}\label{Sect=Semigroup}
We recall preliminaries on semi-groups.

\begin{dfn}   
A map $\Phi: \cM \rightarrow \cM$ is called Markov if it is normal ucp (unital completely positive) and $\varphi \circ \Phi = \varphi$ (where $\varphi$ is the fixed faithful normal state on $\cM$). Through complex interpolation between $\cM$ and $L_1(\cM)$, a Markov map has a contractive $L_2$-implementation given by
\[
\Phi^{(2)}: D_{\varphi}^{\frac{1}{4}} x D_{\varphi}^{\frac{1}{4}} \rightarrow D_{\varphi}^{\frac{1}{4}} \Phi(x) D_{\varphi}^{\frac{1}{4}}.
\]
A Markov map is called {\it KMS-symmetric} if $\Phi^{(2)}$ is self-adjoint. A Markov map is called {\it GNS-symmetric} if $\varphi(\Phi(x)^\ast y) = \varphi(x^\ast \Phi(y))$ for all $x,y \in \cM$. 
 $\Phi$ is called $\varphi$-modular if for every $s \in \mathbb{R}$ we have $\Phi \circ \sigma^{\varphi}_s = \sigma^{\varphi}_s \circ \Phi$.
\end{dfn}

If $\Phi$ is $\varphi$-modular then it is KMS-symmetric if and only if it is GNS-symmetric. 

\begin{dfn}
A family $(\Phi_t)_{t \geq 0}$ is called a semi-group if $\Phi_{s+t} = \Phi_s \circ \Phi_t$ and for every $x \in \cM$ we have $\Phi_t(x) \rightarrow x$ in the strong topology as $t \searrow 0$.  A semi-group $(\Phi_t)_{t \geq 0 }$ is called {\it Markov}, {\it KMS-symmetric} or $\varphi$-{\it modular} if for each $t \geq 0$ the map $\Phi_t$ is respectively Markov, KMS-symmetric or $\varphi$-modular.
\end{dfn}

 By interpolation between $L_1$ and $L_\infty$ we may in fact define $\Phi_t^{(p)}$ as  (the closure of)
\begin{equation}\label{Eqn=PhiLp}
\Phi_t^{(p)}: L_p(\cM) \rightarrow L_p(\cM): D_{\varphi}^{\frac{1}{2p}} x D_{\varphi}^{\frac{1}{2p}}  \mapsto  D_{\varphi}^{\frac{1}{2p}} \Phi_t(x) D_{\varphi}^{\frac{1}{2p}},
\end{equation} 
see \cite[Lemma 7.1]{JungeXuJAMS}. 
If $\Phi_t$ is $\varphi$-modular then for $x$ analytic,
\begin{equation}\label{Eqn=KMSSym}
\begin{split}
& \Phi^{(2)}_t  (x D_{\varphi}^{\frac{1}{2}}) = \Phi^{(2)}_t  ( D_{\varphi}^{\frac{1}{4}}   \sigma_{i/4}^\varphi(x) D_{\varphi}^{\frac{1}{4}}) =    D_{\varphi}^{\frac{1}{4}}  \Phi_t( \sigma_{i/4}^\varphi(x)) D_{\varphi}^{\frac{1}{4}}\\
 =& D_{\varphi}^{\frac{1}{4}}   \sigma_{i/4}^\varphi(\Phi_t(x)) D_{\varphi}^{\frac{1}{4}}
=  \Phi_t(x) D_{\varphi}^{\frac{1}{2}}.
\end{split}
\end{equation}

For $1 \leq p < \infty$ let $A_p \geq 0$ be the unbounded generator of our Markov semi-group, which may be characterized by
\[
\dom(A_p) = \{ \xi \in L_p(\cM) \mid \lim_{t \searrow 0 } t^{-1}( \Phi_t^{(p)}(\xi) - \xi) \textrm{ exists} \}
\]
and for  $\xi \in \dom(A_p),  A_p \xi = \lim_{t \searrow 0 } t^{-1}(\xi - \Phi_t^{(p)}(\xi)  )$.
We have $\exp(-t A_p) = \Phi_t^{(p)}$. We also set,
\[
L_p^\circ( \cM  )  = \left\{ \xi \in L_p(\cM) \mid \lim_{t \rightarrow \infty}  \Phi_t^{(p)}(\xi) = 0   \right\}.
\] 
Note that as $\varphi$ is a normal faithful state, we have an inclusion 
\begin{equation}\label{Eqn=InclusionLp}
 \kappa^{\varphi}_{r,p} :=   (\kappa^{\varphi}_r)^{-1} \circ \kappa^{\varphi}_p: L^p(\cM) \subseteq L^r(\cM): D_{\varphi}^{\frac{1}{2p}}  x  D_{\varphi}^{\frac{1}{2p}} \mapsto D_{\varphi}^{\frac{1}{2r}}  x  D_{\varphi}^{\frac{1}{2r}}, \qquad x \in \cM,
\end{equation}
 whenever  $r \leq p$ and this inclusion is a contractive mapping that intertwines $\Phi_t^{(p)}$ and $\Phi_t^{(r)}$. It follows therefore that $\Dom(A_p) \subseteq \Dom(A_r)$. 
We also set, 
\[
\cM^\circ = \left\{ x \in \cM  \mid \Phi_t(x) \rightarrow 0 \quad \sigma{\rm -weakly}  \right\}.
\]
And for notational convenience $L_\infty^\circ(\cM) = \cM^\circ$.

\begin{lem}\label{Lem=NotEmbedding}
	For $1 \leq r \leq p \leq \infty$ we have  $L_p^\circ(\cM) \subseteq L_r^\circ(\cM)$ for the inclusion \eqref{Eqn=InclusionLp}.
\end{lem}
\begin{proof} 
	Assume $p \not = \infty$. Take  $y \in L_p^\circ(\cM)$ then $\Phi_t^{(p)}(y) \rightarrow 0$. So  $\Phi_t^{(r)} ( \kappa^{\varphi}_{r,p} (  y)) =  \kappa^{\varphi}_{r,p} ( \Phi_t^{(p)} ( y)) \rightarrow 0$ which is equivalent to $\kappa^{\varphi}_{r,p}(y) \in L_r^\circ(\cM)$. Assume $p = \infty$. Take  $y \in \cM^\circ$ so that $\Phi_t^{(p)}(y) \rightarrow 0$   strongly. Then $\Phi_t^{(p)} ( \kappa^{\varphi}_{p,\infty} (  y ) ) = D_{\varphi}^{\frac{1}{2p}} \Phi_t(y) D_{\varphi}^{\frac{1}{2p}} \rightarrow 0$ by \cite[Lemma 1.3]{JungeSherman}.
\end{proof}

\begin{rmk}
 Suppose that the state $\varphi$ is almost periodic, meaning that its modular operator $\nabla_\varphi$ has a complete set of eigenspaces. In this case there is the following averaging trick in order to assure the existence of $\varphi$-modular semi-groups (see e.g. \cite[Theorem 4.15]{OkayasuTomatsu} for a similar argument). 
  By \cite[Lemma 3.7.3]{ConnesENS}   there exists a compact group $\widehat{\Gamma}$ with group homomorphism $\rho: \mathbb{R} \rightarrow \widehat{\Gamma}$ with dense range and   a continuous unitary representation $s \mapsto U_s, s \in \widehat{\Gamma}$ on $B(L_2(\cM))$ such that for $t \in \mathbb{R}$ we get $\nabla_{\varphi}^{it} = U_{\rho(t)}$.
  Let $\Phi$ be a Markov map on $\cM$. Then the map
 \[
 \Phi^{av} =   \int_{\widehat{\Gamma}} {\rm ad}(U_s^\ast) \circ  \Phi  \circ {\rm ad}(U_s) ds
 \] 
 is also Markov. 
 Moreover, it is $\varphi$-modular as
 \[
 \begin{split}
  & \Phi^{av} \circ \sigma_t^{\varphi}
 = \Phi^{av} \circ {\rm ad} (\nabla^{it}_{\varphi})
 = \int_{\widehat{\Gamma}} {\rm ad}(U_s^\ast) \circ  \Phi  \circ {\rm ad}(U_{s+\rho(t)})  ds \\
 = & \int_{\widehat{\Gamma}} {\rm ad}(U_{s}^\ast U_{\rho(t)}  ) \circ  \Phi  \circ {\rm ad}(U_{s })  ds
 = \sigma_t^{\varphi} \circ \Phi^{av}. 
 \end{split}
 \]   
 Similarly, if $(\Phi_t)_{t \geq 0}$ is a Markov semi-group then $(\Phi_t^{av})_{t \geq 0}$ is a Markov semi-group that is moreover $\varphi$-modular.  
\end{rmk}

\subsection{Markov dilations of semi-groups} The following terminology was introduced in  \cite{JungeMei} (see also \cite{Anan} and \cite{RicardDilation}). It forms the crucial condition that is being used  in Junge and Mei their proofs of interpolation results.

\begin{dfn}
A standard Markov dilation of a semi-group $\mathcal{S} = (\Phi_t)_{t\geq 0}$ on a von Neumann algebra $\cM$ with normal faithful state $\varphi$ consists of: (1) A von Neumann algebra $\cN$ with faithful normal state $\varphi_\cN$, (2) an increasing filtration $(\cN_s)_{s \geq 0}$ with $\varphi_\cN$-preserving conditional expectations $\cE_s: \cN \rightarrow \cN_s$, (3) state preserving $\ast$-homomorphisms $\pi_s: \cM \rightarrow \cN_s$ such that
\begin{equation}\label{Eqn=MarkovDilationRelation}
\cE_{s} (\pi_t(x))   = \pi_s(  \Phi_{t-s}(x)  ), \qquad s<t, x \in \cM.
\end{equation}

\end{dfn}

\begin{dfn}
A   reversed   Markov dilation of a semi-group $\mathcal{S} = (\Phi_t)_{t\geq 0}$ on a von Neumann algebra $\cM$ with normal faithful state $\varphi$ consists of: (1) A von Neumann algebra $\cN$ with faithful normal state $\varphi_\cN$, (2) a decreasing filtration $(\cN_s)_{s \geq 0}$ with $\varphi_\cN$-preserving conditional expectations $\cE_s: \cN \rightarrow \cN_s$, (3) state preserving $\ast$-homomorphisms $\pi_s: \cM \rightarrow \cN_s$ such that
\begin{equation}\label{Eqn=RevMarkovDilationRelation}
\cE_{s} (\pi_t(x))   = \pi_s(  \Phi_{s-t}(x)  ), \qquad t<s, x \in \cM.
\end{equation}
\end{dfn}

We call a (standard or reversed) Markov dilation {\it modular} if in their definitions we have moreover
 \begin{equation}\label{Eqn=ModDilation}
 \sigma_t^{\varphi_\cN} \circ \pi_s = \pi_s \circ \sigma_t^{\varphi}, s \geq 0, t \in \mathbb{R}.
 \end{equation}
 Without loss of generality for a standard Markov dilation we may assume that $\cN_s$ is generated by $\pi_t(x), t \leq s, x \in \cM$ and $\cN = \left( \cup_{s \geq 0} \cN_s \right)''$. Then the condition \eqref{Eqn=ModDilation} implies  that $\sigma^{\varphi_{\cN}}$ preserves $\cN_s$ for every $s \geq 0$.  

We typically denote standard/reversed Markov dilations by means of a triple $(\cN_{t}, \pi_t, \cE_t)_{t \geq 0}$. The von Neumann algebra $\cN$ is then implicitly understood as the $\sigma$-weak closure of $\cup_{t \geq 0} \cN_t$.

\begin{dfn} \label{Dfn=AUCt}
	An $L_\infty$-martingale $(x_t)_{t \geq 0}$ in a  von Neumann algebra $\cN$  with faithful normal  state $\psi$ and with filtration $(\cN_t)_{t \geq 0}$  has a.u. continuous path if for every $T > 0, \epsilon > 0$ there exists a projection $e \in \cN$ with $\psi(1 - e) <  \epsilon$ such that $[0, T] \rightarrow \cN: t \mapsto x_t e$ is continuous.
\end{dfn}

We require Lemma \ref{Lem=VanishHp} which was already observed in   \cite[p. 716]{JungeMei} and \cite[p. 637]{JungeMei}. For properties of vector valued $L_p$-spaces we refer to \cite{PisierBook}. 
Let $x = (x_t)_{t \geq 0}$ be a martingale as in Definition \ref{Dfn=AUCt}. Let $2 < p < \infty$. Let $\sigma= \{ t_1, \ldots, t_{n_\sigma} \}$ be a (finite) set of elements $0 < t_1 < \ldots < t_{n_\sigma} < \infty$. We write
\[
\Vert x \Vert_{h_p^d(\sigma)} =  \left( \sum_{t_i \in \sigma} \Vert x_{t_{i+1}} - x_{t_{i}} \Vert_{L_p}^p\right)^{\frac{1}{p}},
\]
and then $\Vert x \Vert_{h_p^d} = \lim_{\sigma, \cU} \Vert x \Vert_{h_p^d(\sigma)}$ for any ultrafilter containing the filter base of tails. This yields a norm, which is independent of the choice of ultrafilter \cite{JungePerrin}. Note that the  $h_p^d(\sigma)$-norm is just the  $L_p(\ell_p(\sigma))$-norm  \cite{PisierBook} of the martingale difference sequence $d_i(x) =  x_{t_{i+1}} - x_{t_{i}}$.
It follows straight from the definitions that if $\mathcal{Q}$ is a von Neumann subalgebra of $\mathcal{N}$ with expectation $\cE_{\mathcal{Q}}$ satisfying for all $t \geq 0$,  $\cE_{\mathcal{Q}} \circ \cE_t = \cE_t \circ \cE_{\mathcal{Q}}$. Then for every martingale $x = (x_t)_{t \geq 0}$ in $\cN$ we get 
\begin{equation}\label{Eqn=HardyExpected}
\Vert \cE_{\mathcal{Q}}(x) \Vert_{h_p^d} \leq \Vert x \Vert_{h_p^d}.
\end{equation}

\begin{lem}\label{Lem=VanishHp}
	If a martingale $x = (x_t)_{t \geq 0}$ has a.u. continuous path then $\Vert x \Vert_{h_p^d} = 0$ for all $p > 2$. 
\end{lem}
\begin{proof}
 	We use the notation of Definition \ref{Dfn=AUCt}. 
	By Doob's inequality \cite{JungeCrelle} for every $2 < p < \infty$ and $T> 0$ there exists a continuous function $f: [0,T] \rightarrow \cN$ and an element $a \in L_p(\cN)$ such that $x_t = f(t) a$. 
	  Then taking the ultralimit over all finite subsets $\sigma \subseteq [0,T]$ we get $\Vert x \Vert_{L_p(\ell_\infty^c(\sigma))}  \rightarrow 0$.
	  By interpolation 
  \[
  \Vert x \Vert_{h_p^d(\sigma)} \leq \Vert d_j(x) \Vert_{L_{p}(\ell_\infty^c(\sigma))}^\theta \Vert d_j(x) \Vert_{L_{p}(\ell_2^c(\sigma)}^{1-\theta},
  \]
  with $\theta = p/2$.  Let $\sigma = \{ t_1 < \ldots < t_n \}$ be a finite subset of $[0,T]$.  Set $d_j(x) = x_{t_{j+1}} -  x_{t_{j}}$.
   The norm $\Vert (d_j(x))_j \Vert_{L_{p}(\ell_2^c(\sigma)}$ can be upper estimated by the norm  $\Vert x \Vert_p$ by the Burkholder-Gundy inequality \cite[Theorem 6.4]{HaagerupJungeXu} and in particular is uniformly bounded in $\sigma$. Then as we already showed that $\Vert d_j(x) \Vert_{L_{p}(\ell_\infty^c(\sigma))} \rightarrow 0$ we conclude. 

\end{proof}

Because modular Markov dilations are state preserving homomorphisms, they extend to maps
\[
\pi_s^{(p)}: L_p(\cM) \rightarrow  L_p(\cN_s): D_{\varphi}^{\frac{1}{2p}} x D_{\varphi}^{\frac{1}{2p}} \rightarrow    D_{\varphi_{\cN}}^{\frac{1}{2p}}  \pi_s(x) D_{\varphi_{\cN}}^{\frac{1}{2p}}, \qquad x \in \cM, s \geq 0.
\]
 These are $\cM$-$\cM$ bimodule maps in the sense that $\pi_s(x) \pi_s^{(p)}(y) \pi_s(z) = \pi_s^{(p)}(xyz)$ for $x,z \in \cM$ and $y \in L_p(\cM)$.
 
  We shall need a notion of almost uniform continuity of Markov dilations. These notions were considered in   \cite{JungeMei} (see also \cite{JungeMeiAJM}) and play an important role for embeddings of various BMO-spaces. Our notion {\it differs} from what is used in \cite[p. 725]{JungeMei}, which assumes a.u. continuity of two martingales $m(f)$ and $n(f)$. But actually the proof of the interpolation result in the first statement of \cite[Theorem 5.2 (ii)]{JungeMei} only uses a.u. continuity of the  martingale $m(f)$, which is what we need (the second statement of \cite[Theorem 5.2 (ii)]{JungeMei} requires more).

 \begin{dfn}\label{Dfn=AUCtPath}
 	A reversed Markov dilation $(\cN_{t}, \pi_t, \cE_t)_{ t \geq 0}$ for a Markov semi-group $\cS = (\Phi_t)_{t\geq0}$ on a von Neumann algebra $\cM$ has  a.u. continuous path if there exists a $\sigma$-weakly dense subset $B \subseteq \cM$  such that for all $x \in B$ the $L_\infty$-martingale
 	\begin{equation}\label{Eqn=mMartingale}
 	m(x) = (  m_t(x)   )_{t \geq 0} = ( \pi_t   \circ \Phi_t(x) )_{t \geq 0}.
 	\end{equation}
 	has a.u. continuous path.  
 \end{dfn}

\begin{rmk}
In the work in progress \cite{JRS} it is proved that Markov semi-groups on finite von Neumann algebras always admit a standard (as well as reversed) Markov dilation with a.u. continuous path. 
\end{rmk}

\section{Semi-group BMO for $\sigma$-finite von Neumann algebras} \label{Sect=BMO}

In this section we generalize some of the interpolation results from \cite{JungeMei}, in particular Theorem \ref{Thm=JungeMei}, for finite von Neumann algebras to arbitrary $\sigma$-finite von Neumann algebras. 

Throughout this section we let  $\mathcal{S} = (\Phi_t)_{t \geq 0}$ be a Markov semi-group on a $\sigma$-finite von Neumann algebra $\cM$ with fixed normal faithful state $\varphi$. In order to do reduction we must assume later that $\mathcal{S}$ is $\varphi$-modular. Furthermore in order to  interpret BMO-spaces (see Section \ref{Sect=BMO}) as interpolation spaces we must assume that $\cS$ is GNS-symmetric  (which in case the semi-group is $\varphi$-modular is equivalent to being KMS-symmetric).

\subsection{The Haagerup reduction method}
Let $\sfG = \cup_{n \in \mathbb{N}}  \frac{1}{n} \mathbb{Z}$ equipped with the discrete topology.   We set $\cR = \cM \rtimes_{\sigma^{\varphi}} \sfG$ which is the subalgebra of $\cM \otimes \cB(\ell_2(\sfG))$ generated by operators
\begin{equation}\label{Eqn=Crossed}
l_g = 1 \otimes \lambda_g, \qquad \pi_\varphi(x) = \sum_{g \in \sfG}  \sigma^{\varphi}_{-g}(x) \otimes e_{g, g}, \qquad g \in \sfG, x \in \cM.
\end{equation}
The map $\pi_\varphi$ identifies $\cM$ as a subalgebra of $\cR$ and hence we often omit it.
For every $\gamma \in \widehat{\sfG}$ there exists  an automorphism $\theta_\gamma: \cR \rightarrow \cR$ called the {\it dual action} that is determined by $\theta_\gamma(\pi_{\varphi}(x)) = \pi_\varphi(x), \theta_\gamma( l_g) = \langle \gamma, g \rangle_{\widehat{\sfG}, \sfG} l_g$ with $x \in \cM, g \in \sfG$. There exists a normal conditional expectation $\cE_{\cM}: \cR \rightarrow \pi_\varphi(\cM) \simeq \cM$ that is given by 
\begin{equation} \label{Eqn=CEVcross}
\cE_{\cM}(x) =  \int_{\gamma \in \widehat{\sfG}} \theta_{\gamma}(x) d\gamma, \qquad x \in \cR.
\end{equation}
 We set $\widetilde{\varphi} = \varphi \circ \pi_\varphi^{-1} \circ \cE_{\cM}$, which is a normal faithful state on $\cR$ that restricts to $\varphi$ on $\cM$. We define   $b_n = -i \log(\lambda_{2^{-n}})$ where we use the principal branch of the logarithm so that $0 \leq \Im(\log(z)) < 2\pi$. Then set $a_n = 2^n b_n, h_n = e^{-a_n}$ and
\[
\widetilde{\varphi}_n  = h_n^{\frac{1}{2}} \widetilde{\varphi_n} h_n^{\frac{1}{2}}, \qquad \cR_n := \cR_{\widetilde{\varphi}_n},
\]
Here  $\cR_{\widetilde{\varphi}_n} :=  \{ x \in \cR \mid \sigma^{\widetilde{\varphi}_n}_t(x) = x \}$ is the centralizer of $\widetilde{\varphi}_n$. By construction the operator $h_n$ is boundedly invertible. Furthermore,
\begin{equation}\label{Eqn=Varphi}
D_{\widetilde{\varphi}}^{it}  h_n D_{\widetilde{\varphi}}^{-it} = h_n, \quad {\rm and } \quad  D_{\widetilde{\varphi}_n}^{it}  h_n D_{\widetilde{\varphi}_n}^{-it} = h_n.
\end{equation}
Now we recall the following theorem from \cite{HaagerupJungeXu} (see also \cite[Section 7]{CPPR} for the weight case), which is known as the reduction method.

\begin{thm}\label{Thm=Reduction}
With the above notation we have:
\begin{enumerate}
\item Each $\cR_n$ is finite with normal faithful trace $\widetilde{\varphi}_n$.
\item  There exist  normal  conditional expectations  $\cE_n: \cR \rightarrow \cR_n$ such that $\widetilde{\varphi} \circ \cE_n = \widetilde{\varphi}$ and $\sigma^{\widetilde{\varphi}}_t \circ \cE_n   = \cE_n \circ \sigma^{\widetilde{\varphi}}_t$ for all $t \in \mathbb{R}$.
 \item For each $x \in \cR$ we have $\cE_n(x) \rightarrow x$ in the $\sigma$-strong topology.
\end{enumerate}
\end{thm}

The following lemma is standard. We included a sketch of the proof for convenience of the reader.
\begin{lem}\label{Lem=MarkovExtension} 
Let $\Phi$ be a $\varphi$-modular Markov map on $\cM$. Then there exists a unique normal $\widetilde{\varphi}$-modular extension $\widetilde{\Phi}$ on $\cR$ such that
\begin{equation}\label{Eqn=PhiTilde}
\widetilde{\Phi}( \pi_\varphi(  x) \lambda_g ) = \pi_\varphi( \Phi(x)) \lambda_g,   \qquad x \in \cR, g \in \widehat{\mathsf{G}}.
\end{equation}
In particular we have
\begin{equation}\label{Eqn=HCommute}
\widetilde{\Phi}(h_n^{it} \pi_{\varphi}(x) h_n^{-it}) = h_n^{it}  \widetilde{\Phi}( \pi_{\varphi}(x) )h_n^{-it}, x \in \cM.
\end{equation}
Moreover if $\Phi$ is Markov then so is $\widetilde{\Phi}$ and if $(\Phi_t)_{t \geq 0}$ is a Markov semi-group  then so is $(\widetilde{\Phi}_t)_{t\geq 0}$ for both $\widetilde{\varphi}$ and $\widetilde{\varphi}_n$.  If $(\Phi_t)_{t \geq 0}$ is KMS-symmetric, then so is $(\widetilde{\Phi}_t)_{t \geq 0}$ for both $\widetilde{\varphi}$ and $\widetilde{\varphi}_n$.
\end{lem}
\begin{proof}
As $\cR = \cM \rtimes_{\sigma^{\varphi}} \mathsf{G} \subseteq \cM \otimes \cB(\ell_2(\mathsf{G}))$. We let $\widetilde{\Phi}$ be the restriction of $\Phi \otimes \id_{ \cB(\ell_2(\mathsf{G}))}$ to $\cR$. Using that $\Phi$ commutes with the modular group of $\varphi$ \eqref{Eqn=PhiTilde} follows. If $\Phi$ is Markov then 
for $x \in \cM, g \in \sfG$, 
\[
\begin{split}
& \widetilde{\varphi} \circ \widetilde{\Phi} ( \pi_{\varphi}(x)  \lambda_g)
=  \widetilde{\varphi}   ( \pi_{\varphi}(\Phi(x) )  \lambda_g)
=  \varphi \circ \cE_{\cM}(  \pi_{\varphi}(\Phi(x) )  \lambda_g ) \\
= &     \varphi( \Phi(x) ) \delta_{g,0}
=  \varphi( x ) \delta_{g,0}
= \widetilde{\varphi} ( \pi_{\varphi}(x)  \lambda_g).
\end{split}
\]
So $\widetilde{\Phi}$ is Markov. 
  As $h_n$ is contained in $1 \otimes \mathcal{L}(\sfG)$ we have that $\widetilde{\Phi}(h_n^{\ast} h_n) = h_n^\ast h_n =  \widetilde{\Phi}(h_n)^\ast  \widetilde{\Phi}(h_n)$. So $h_n$ is in the multiplicative domain of $\widetilde{\Phi}$ \cite[Proposition 1.5.6]{NateTaka} and so 
$\widetilde{\Phi}(h_n^{\frac{1}{2}} y h_n^{\frac{1}{2}} ) = h_n^{\frac{1}{2}} \widetilde{\Phi}( y  ) h_n^{\frac{1}{2}}$. It follows that   $\widetilde{\Phi}$ is   Markov for $\widetilde{\varphi}_n$. Also $\widetilde{\Phi}(h_n^{-it} h_n^{it}) = 1 = h_n^{-it}h_n^{it} = \widetilde{\Phi}(h_n^{-it})  \widetilde{\Phi}(h_n^{it})$, so that also $h_n^{it}$ is in the multiplicative domain of $\widetilde{\Phi}$ and so \eqref{Eqn=HCommute} follows from \cite[Proposition 1.5.6]{NateTaka}.

Furthermore, since $\Phi_t$ is strongly continuous $\widetilde{\Phi}_t$ is strongly continuous. 
 As the modular group $\sigma^{\widetilde{\varphi}}$ is determined by $\sigma^{\widetilde{\varphi}}_t(\pi_\varphi(x)) = \pi_\varphi( \sigma_t^\varphi(x) ), x \in \cM$ and $\sigma^{\widetilde{\varphi}}_t(l_g) = l_g, g \in \sfG$ it follows that $\widetilde{\Phi}_t \circ \sigma^{\widetilde{\varphi}}_s =  \sigma^{\widetilde{\varphi}}_s \circ \widetilde{\Phi}_t$.
From the definition one finds that $\widetilde{\varphi}(\widetilde{\Phi}_t(x) y) = \widetilde{\varphi}(x \widetilde{\Phi}_t(y))$ which for   $\widetilde{\varphi}$-modular semi-groups yields that the semi-group is KMS-symmetric (see \eqref{Eqn=KMSSym}).
\end{proof}

From this point let $\widetilde{\cS} = (\widetilde{\Phi}_t)_{t \geq 0}$ be the extension of the Markov semi-group of Theorem \ref{Lem=MarkovExtension} of a  prefixed $\varphi$-modular Markov semi-group $\mathcal{S} = (\Phi_t)_{ t \geq 0}$.

\subsection{Reducing Markov dilations}

In this section we show that Markov dilations and a.u. continuity behaves well with reduction. 

\begin{prop}\label{Prop=SExtension}
Suppose that $\mathcal{S}$ is $\varphi$-modular and admits a standard (resp. reversed) $\varphi$-modular Markov dilation. Then the semi-group $\widetilde{\mathcal{S}}$ admits a standard (resp. reversed) $\widetilde{\varphi}$-modular Markov dilation. Moreover, if the  reversed Markov dilation of $\mathcal{S}$  has a.u. continuous path, then the reversed Markov dilation of $\widetilde{\mathcal{S}}$ may be chosen to have a.u. continuous path. 
\end{prop}
\begin{proof}
 As before write $\cS = (\Phi_t)_{t \geq 0}$ for the semi-group and $\widetilde{\cS} = (\widetilde{\Phi}_t)_{t \geq 0}$ for the crossed product extension as in Lemma \ref{Lem=MarkovExtension}.

\vspace{0.3cm}

\noindent {\it Part 1: Dilations.}  
Let $(\cN_s, \pi_s, \cE_s)_{s \geq 0}$ be a $\varphi$-modular reversed Markov dilation for $\cS$ with respect to a normal faithful state $\psi$ on $\cN$. Let $\cO = \cN \rtimes_{\sigma^{\psi_\cN}} \sfG$ and $\cO_s = \cN_s \rtimes_{\sigma^{\psi_\cN}} \sfG$ and equip it with the dual weight $\widetilde{\psi} = \psi \circ \pi_\psi^{-1} \circ \int_{\widehat{\sfG}} \theta_\gamma d\gamma$. Because $\sigma^{\psi}_t \circ \pi_s = \pi_s \circ \sigma^{\varphi}_t$ it follows that $\pi_s$ extends uniquely to a normal map $\widetilde{\pi}_s: \cR \rightarrow \cO$ that intertwines the modular groups of $\sigma^\psi$ and $\sigma^{\varphi}$. Similarly because for $\varphi_{\cN}$-preserving conditional expectations we have $\cE_s \circ \sigma^{\psi}_t = \sigma_t^\psi \circ \cE_s, s \geq 0, t \in \mathbb{R}$ we get conditional expectations $\widetilde{\cE}_s: \cO \rightarrow \cO_s$. In particular $(\cO_s)_{s \geq 0}$ is filtered as the operators $\pi_{\psi}(x), x \in \cup_{s \geq 0} \cN_s, l_t, t \in \sfG$ are dense in $\cO$. We claim that $( \cO_s,  \widetilde{\pi}_s,  \widetilde{\cE}_s)_{s \geq 0}$ is a reversed Markov dilation.

For $g \in \sfG$ let $l_g^{\cR} \in \cR$ and  $l_g^{\cO} \in \cO$ be the operators $l_g$ of \eqref{Eqn=Crossed} in these respective von Neumann algebras.
It follows  from the relations $\widetilde{\Phi}_t \circ \pi_\varphi = \pi_\varphi \circ \Phi_t$, $\widetilde{\pi}_s \circ \pi_\varphi = \pi_{\psi} \circ \pi_s$ and  $\pi_{\psi} \circ  \cE_{s} = \widetilde{\cE}_{s} \circ \pi_{\psi}$ that for $x\in \cM$, $t < s$, 
\begin{equation}\label{Eqn=IntertwineI}
\begin{split}
& \widetilde{\pi}_s \circ \widetilde{\Phi}_{s-t} ( \pi_{\varphi}(x)  l_g^{\cR} )
= \widetilde{\pi}_s \circ \pi_{\varphi} (    \Phi_{s-t}(x) )l_g^{\cO} 
= \pi_{\psi} \circ  \pi_s \circ  \Phi_{s-t}(x)l_g^{\cO} \\
= & \pi_{\psi} \circ  \cE_{s} \circ \pi_t(x)l_g^{\cO}
=  \widetilde{\cE}_{s} \circ \widetilde{\pi}_t (\pi_{\varphi}(x) l_g^{\cR}).
\end{split}
\end{equation}
Therefore  \eqref{Eqn=MarkovDilationRelation} follows by density.

This proves the first statement for reversed Markov dilations, for standard Markov dilations the proof is similar.

\vspace{0.3cm}

\noindent {\it Part 2: A.u. continuity of the paths.} Suppose now that the reversed Markov dilation   $(\cN_s, \pi_s, \cE_s)_{s \geq 0}$  in part 1 has a.u. continuous path.   By Definition \ref{Dfn=AUCtPath} there exists a $\sigma$-weakly dense subspace $B \subseteq \cM$ such that for $x \in B, T>0, \epsilon >0$  there is a projection $e \in \cN$   with $\psi(1-e)  <\epsilon$ such that the map $[0,T] \ni t \mapsto m_t(x) e$ is continuous with $m_t(x) := \pi_t  (\Phi_t (x))$. Let $\widetilde{e}  = \pi_{\psi} (e)$. Then we have for $g \in \sfG$ that
\[
\begin{split}
& \widetilde{\pi}_t   (  \widetilde{\Phi}_t ( l_g \pi_\varphi (x)  )) \widetilde{e}  
=   l_g \widetilde{\pi}_t   ( \widetilde{\Phi}_t   ( \pi_\varphi(x)  ) )  \pi_{\psi} (e) \\
= & l_g  \pi_{\psi}  \left( \pi_t  (  \Phi_t   (   x )  )   \right)   \pi_{\psi} (e)   
=   l_g  \pi_{\psi}   \left( \pi_t   (  \Phi_t   (   x )  )  e \right).   
\end{split}
\]
So if we put $\widetilde{m}_t( y ) :=  \widetilde{\pi}_t (  \widetilde{\Phi}_t( y ))$ we see that 
\[
[0, T] \ni t \mapsto \widetilde{m}_t(   l_g \pi_\varphi(x)     ) \widetilde{e} = l_g  \pi_{\psi} \left(  m_t( x ) e \right).  
\] 
is continuous as  $[0, T] \ni t \mapsto  m_t( x ) e$ is continuous. As the span of $l_g\pi_\varphi(x), x \in B, g \in \sfG$ is $\sigma$-weakly dense in $\mathcal{O}$ this concludes the second claim. 
\end{proof}

\subsection{Semi-group BMO and interpolation structure}\label{Sect=BMO}    For $x \in \cM$ we define the column BMO-semi-norm
\[
\Vert x \Vert_{\BMO^c_{\mathcal{S}}} = \sup_{t} \Vert \Phi_t(x)^\ast \Phi_t(x) - \Phi_t(x^\ast x) \Vert^{\frac{1}{2}}.
\]
Then set the row BMO-semi-norm and the BMO-semi-norm  by
\[
\Vert x \Vert_{\BMO^r_{\mathcal{S}}} = \Vert x^\ast \Vert_{\BMO^c_{\mathcal{S}}}, \qquad \Vert x \Vert_{\BMO_{\mathcal{S}}} = \max(  \Vert x \Vert_{\BMO^c_{\mathcal{S}}},  \Vert x \Vert_{\BMO^r_{\mathcal{S}}}).
\]
As proved in \cite[Proposition 2.1]{JungeMei} these assignments are indeed semi-norms. To proceed further to interpolation we need to treat normed spaces instead and  we need to identify BMO-spaces as subspaces of $L_1(\cM)$. We can do this using GNS-symmetry and  modularity of   Markov semi-groups. Note that in \cite{JungeMei} KMS/GNS-symmetry is also part of the standard assumptions on the semi-groups. 

\begin{lem}\label{Lem=BMONilspace}
We have,
	\[
	\left\{ x \in \cM  \mid \Vert x \Vert_{ \BMO_{\cS} } = 0   \right\}
	\supseteq \left\{ x \in \cM  \mid \forall t \geq 0: \:  \Phi_t (x) = x   \right\}.
	\]
	Moreover, 	if $\cS$ is GNS-symmetric then we have equality of these sets. In particular  on  $\cM^\circ$ the $\BMO_{\mathcal{S}}$-semi-norm is actually a norm.
\end{lem}
\begin{proof}
		\noindent $\supseteq$. For each $t$ the space of fixed points for $\Phi_t$ is a $\ast$-algebra, see \cite[Remark 7.3]{JungeXuJAMS}. This shows that if $\Phi_t(x) = x$ we also have that
	\[
	\Phi_t(x^\ast x) - \Phi_t(x)^\ast \Phi_t(x) = x^\ast x - x^\ast x = 0,
	\]
	and similarly with $x$ replaced by $x^\ast$. That is $\Vert x \Vert_{ \BMO_{\cS} } = 0$. 
	
	\noindent $\subseteq$. Assume $\cS$ is GNS-symmetric.  If $\Vert x \Vert_{ \BMO_{\cS} } = 0$ (in particular both the row and column BMO-semi-norm is 0) then by \cite[Proposition 1.5.6]{NateTaka} we see that $x$ is in the multiplicative domain of $\Phi_t$ for every $x \in \cM$. We then get for $y \in \cM$ that
	\[
	\varphi(yx) = \varphi(\Phi_t(y x))
	= \varphi(\Phi_t(y) \Phi_t(x)) 
	= \varphi(y \Phi_{2t}(x)),
	\]
	where the last equality uses $\Phi_t$ is GNS-symmetric. This implies that $\Phi_{2t}(x) = x$ for all $t \geq 0$. 
	
	Finally, take $x \in \cM^\circ$ so $\Phi_t(x) \rightarrow 0$ $\sigma$-weakly. Then, if $\Vert x \Vert_{\BMO_{\cS}} = 0$ we get by this lemma that for all $t \geq 0$ we have $\Phi_t(x) = x$ so that $x = 0$.  
	
\end{proof}

 If $\cS$ is $\varphi$-modular GNS-symmetry of $\cS$ is equivalent to  KMS-symmetry. We prefer to include the KMS-symmetry as part of our statements as all embeddings and interpolation structures are defined with respect to symmetric embeddings. Assume now that $\cS$ is $\varphi$-modular and KMS-symmetric. 
We write $\BMO^\circ_{\mathcal{S}}$ for the completion of $\cM^\circ$ equipped with respect to the $\BMO^\circ_{\mathcal{S}}$-norm. We denote  $\BMO^\circ_{\mathcal{S}}(\cM)$ in case we explicitly want to distinguish the von Neumann algebra.

We now turn $\BMO^\circ_{\mathcal{S}}$ and into the framework of  compatible couples of Banach spaces, see \cite{BerghLofstrom}.  Here we really need to restrict ourselves to $\BMO^\circ_{\mathcal{S}}$ and not just $\cM$ with the $\BMO_{\mathcal{S}}$-norm.

\begin{lem}\label{Lem=SemiGroupTechnicalThing}
	Suppose that $A_2 \geq 0$ is a positive self-adjoint operator on a Hilbert space $H$ so that $\Phi_t^{(2)} = \exp(-tA_2)$ is a semi-group of positive contractions on $H$. Suppose that for $\xi \in H$ we have that $\Phi_t^{(2)} \xi \rightarrow 0$ weakly as $t \rightarrow \infty$. Then in fact $\Phi_t^{(2)} \xi \rightarrow 0$ in the norm of $H$.
\end{lem} 
\begin{proof}
	Take a spectral resolution $A_2 = \int_{0}^\infty  \lambda dE_A(\lambda)$. Let $p_0$ be the kernel projection of $A_2$ and let $p_1 = 1 - p_0$. Then $\Phi_t^{(2)} \xi \rightarrow 0$ weakly  implies that $p_0 \xi = 0$. Now let $p$ be a spectral projection of $A_2$ of an interval $[\lambda_0, \infty]$ such that $\Vert (1- p) \xi \Vert_H \leq \epsilon$. Choose $t_0 \geq 0$ such that for $t \geq t_0$ we have $\Vert \exp(-t A_2) p \xi \Vert_H \leq \epsilon$.   Then we see $\Vert   \exp(-t A_2) \xi   \Vert_H \leq  \Vert   \exp(-t A_2) p \xi   \Vert_H + \Vert   \exp(-t A_2) (1-p) \xi   \Vert_H  \leq 2 \epsilon$.
\end{proof}

\begin{lem}\label{Lem=BMOinL1}
	  Let $\cS = (\Phi_t)_{t \geq 0}$ be a $\varphi$-modular, KMS-symmetric, Markov semi-group. Consider BMO-spaces as defined above. 
	We have $\BMO_{\cS}^\circ   \subseteq L_1^\circ(\cM)$ through an extension of the embedding $x \mapsto D^{\frac{1}{2}}_{\varphi} x D^{\frac{1}{2}}_{\varphi}$. Moreover, for $x \in \cM^\circ \subseteq \BMO_{\cS}^\circ$ we have 
	\[
	\Vert D^{\frac{1}{2}}_{\varphi} x D^{\frac{1}{2}}_{\varphi} \Vert_1 \leq \Vert x \Vert_{\BMO_{\cS}^c} \qquad {\rm  and } \qquad \Vert D^{\frac{1}{2}}_{\varphi} x D^{\frac{1}{2}}_{\varphi} \Vert_1 \leq \Vert x \Vert_{\BMO_{\cS}^r}.
	\]
\end{lem}
\begin{proof}
	$\Phi_t^{(2)}$ is a semi-group of positive contractions on $L_2(\cM)$. Further,   $\varphi$-modularity of $\cS$ implies that $\Phi^{(2)}_t  (x D_{\varphi}^{\frac{1}{2}}) =  \Phi_t(x) D_{\varphi}^{\frac{1}{2}}$ by \eqref{Eqn=KMSSym}. 
Take $x \in \cM^\circ$ so that for $y \in \cM$
	\[
	\lim_{t \rightarrow \infty}  \langle \Phi_t(x) D_{\varphi}^{\frac{1}{2}}, y D_{\varphi}^{\frac{1}{2}} \rangle
	=
	\lim_{t \rightarrow \infty}   \varphi(y^\ast \Phi_t(x)) \rightarrow 0.
	\]
	This shows that $\Phi_t(x) D_{\varphi}^{\frac{1}{2}} \rightarrow 0$ weakly. 
	By Lemma \ref{Lem=SemiGroupTechnicalThing}  then
	$\Vert \Phi_t(x) D_{\varphi}^{\frac{1}{2}} \Vert_2  \rightarrow 0$.  Writing $x = u \vert x \vert$ for the polar decomposition we therefore see that 
	\[
	\begin{split}
	& \Vert D_{\varphi}^{\frac{1}{2}}    x   D_{\varphi}^{\frac{1}{2}} \Vert_1^2 = 
	\Vert D_{\varphi}^{\frac{1}{2}}    u \vert x \vert   D_{\varphi}^{\frac{1}{2}} \Vert_1^2 
	\leq 	\Vert D_{\varphi}^{\frac{1}{2}}    u \Vert_2^2 \Vert \vert x \vert   D_{\varphi}^{\frac{1}{2}} \Vert_2^2 \leq \Vert \vert x \vert D_{\varphi}^{\frac{1}{2}} \Vert_2^2 \\
	= &  \limsup_{t \rightarrow  \infty}   \varphi( x^\ast x ) - \Vert \Phi_t(x) D_{\varphi}^{\frac{1}{2}} \Vert_2^2 =
	\limsup_{t \rightarrow  \infty }  \varphi(  \Phi_t( x^\ast x )  -  \Phi_t(x^\ast) \Phi_t(x)  ) \\
	\leq &
	\limsup_{t \rightarrow  \infty } \Vert  \Phi_t( x^\ast x ) - \Phi_t(x^\ast) \Phi_t(x)   \Vert =  
	\sup_{t \geq 0 } \Vert  \Phi_t( x^\ast x ) - \Phi_t(x^\ast) \Phi_t(x)   \Vert   
	\end{split}
	\] 
	This shows that $\Vert D_{\varphi}^{\frac{1}{2}}    x   D_{\varphi}^{\frac{1}{2}} \Vert_1 \leq \Vert x \Vert_{\BMO^c_{\cS}}$, which yields the claim. As $\Vert D_{\varphi}^{\frac{1}{2}}    x   D_{\varphi}^{\frac{1}{2}} \Vert_1 = \Vert D_{\varphi}^{\frac{1}{2}}    x^\ast   D_{\varphi}^{\frac{1}{2}} \Vert_1$ we also get that  $\Vert D_{\varphi}^{\frac{1}{2}}    x   D_{\varphi}^{\frac{1}{2}} \Vert_1 \leq \Vert x \Vert_{\BMO^r_{\cS}}$.
	Note that $D_{\varphi}^{\frac{1}{2}}  x D_{\varphi}^{\frac{1}{2}} \in L_1^{\circ}(\cM)$ by Lemma \ref{Lem=NotEmbedding}. 
	 Then in particular, as by construction $\cM^\circ$ is dense in $\BMO_{\cS}^\circ$, we get  $\BMO_{\cS}^\circ \subseteq L_1^\circ(\cM)$ through the embedding of the lemma. 
\end{proof}

We denote the embedding extending $\cM^\circ \ni x \mapsto D_{\varphi}^{\frac{1}{2}}x D_{\varphi}^{\frac{1}{2}} \in L_1^{\circ}(\cM)$ of Lemma \ref{Lem=BMOinL1} by
\begin{equation}\label{Eqn=BMOEmbedding}
\kappa_{\BMO}^{\varphi}: \BMO_{\cS}^\circ \hookrightarrow L_1^\circ(\cM). 
\end{equation}
This shows in particular that $(\BMO_{\cS}^\circ, L_1^\circ(\cM))$ forms a compatible couple of Banach spaces.  

\begin{rmk}
	We did not consider compatible couples for the case that $\varphi$ is an arbitrary normal, semi-finite, faithful weight; neither this seems obvious. For $L_p$-spaces such interpolation structures were explored in \cite{TerpII} and \cite{Izumi}. 
\end{rmk}

We recall the following tracial theorem which we will generalize to the non-tracial setting in this paper. 

\begin{thm}[Theorem 5.2.(ii) of \cite{JungeMei}]\label{Thm=JungeMei}
	Assume that $\cM$ is finite and $\varphi$ is a normal faithful tracial state. Suppose that $\cS$ is a KMS-symmetric Markov semi-group.  Assume that $\mathcal{S}$ admits a reversed Markov dilation with  a.u. continuous path. Then for all $1 \leq p < \infty, 1 < q < \infty$ we have
	\[
	[\BMO_{\mathcal{S}}^\circ(\cM),  L_p^\circ(\cM)]_{\frac{1}{q}}^\varphi \approx_{pq} L_{pq}^\circ(\cM).
	\]
	Here $\approx_{pq}$ means complete isomorphism of operator  spaces with complete norm of the isomorphism and its inverse bounded by a constant times $pq$.
\end{thm}

\begin{rmk}\label{Rmk=HpCondition}
 As noted already in \cite[p. 716, after Lemma 4.1]{JungeMei}, in Theorem \ref{Thm=JungeMei} the condition that $\mathcal{S}$ has a.u. continuous path may be replaced by the weaker condition (see Lemma \ref{Lem=VanishHp}) that  there exists a $\sigma$-weakly dense subset $B \subseteq \cM$ such that for every $2 \leq p < \infty$  the martingale $m(x), x \in B$ defined in \eqref{Eqn=mMartingale} has the property that  $\Vert m(x) \Vert_{h_{p}^d} = 0$.  
\end{rmk}

\subsection{Interpolation for $\sigma$-finite BMO-spaces}\label{Sect=ReductionInt}

We explicitly record the following lemma here, which is an immediate consequence of complex interpolation, c.f. \cite{BerghLofstrom}. Recall that a subspace $Y$ of a Banach space $X$ is called 1-complemented if there is a norm 1 projection $p: X \rightarrow Y$. 

\begin{lem}\label{Lem=OneCompl}
Let $(X_1, X_2)$ be a compatible couple of Banach spaces. Let $Y_i \subseteq X_i$ be $1$-complemented subspaces. Then $(Y_1, Y_2)_{\theta}$  is a 1-complemented subspace of $(X_1, X_2)_{\theta}$.
\end{lem}
 
 Next we prove that the inclusions of BMO-spaces we need to consider in the proof of Theorem \ref{Thm=InterpolationBMO} are 1-complemented. Both proofs are based on finding Stinespring dilations of the semi-group and the conditional expectation that  `commute' in some sense, c.f. \eqref{Eqn=House} and \eqref{Eqn=CommuteE2}. 
 
\begin{prop}\label{Prop=BMOComplementM}
Let $\cS$ be a $\varphi$-modular Markov semi-group. 	We have that  $\BMO^\circ_{\mathcal{S}}(\cM)$ is an isometric 1-complemented subspace of $\BMO^\circ_{\widetilde{\mathcal{S}}}(\cR)$.
\end{prop}
\begin{proof}
	As $\widetilde{\cS}$ restricts to $\cS$ on $\cM$ it follows straight from the definition of BMO-spaces that $\BMO^\circ_{\mathcal{S}}(\cM)$ is an isometric subspace of $\BMO^\circ_{\widetilde{\mathcal{S}}}(\cR)$.
	 
	We now prove that the conditional expectation $\cE_{\cM}$ provides a norm 1 projection $\BMO^\circ_{\widetilde{\mathcal{S}}}(\cR) \rightarrow \BMO^\circ_{\mathcal{S}}(\cM)$. 	
	For every $t$ we may take a Stinespring dilation for the ucp map $\Phi_t$. That is, there exist a Hilbert space $H$, a contractive map $V_t: L_2(\cM) \rightarrow H$ and a representation $\pi_t: \cM \rightarrow B(H)$ such that $\Phi_t(x) = V_t^{\ast} \pi_t ( x) V_t$. 
	We take amplifications $\widetilde{V}_t = 1_{\ell_2(\sfG)} \otimes V_t: \ell_2(\sG) \otimes L_2(\cM) \rightarrow \ell_2(\sG) \otimes H$ and $\widetilde{\pi}_t = 1_{B(\ell_2(\sG))} \otimes \pi_t$.
	Then $\widetilde{\Phi}_t(x) = \widetilde{V}_t^\ast \widetilde{\pi}_t(x) \widetilde{V}_t, x \in \cR$.

	 For $\gamma \in \widehat{\sG}$ set $W_\gamma: \ell_2(\sG) \rightarrow \ell_2(\sG)$ by $(W_\gamma \xi)(s) = \langle \gamma, s \rangle \xi(s)$. Define a partial isometry
	\[
	W: \ell_2(\sG) \rightarrow  L_2(\widehat{\sG},  \ell_2( \sG ) ): 
	\xi \mapsto \left( \gamma \mapsto W_\gamma \xi \right). 
	\] 
	We may naturally view $W$ as a map $\ell_2(\sG) \rightarrow  \ell_2(\sG) \otimes  L_2(\widehat{\sG})$. We extend this map to a map $\widetilde{W}: \ell_2(\sG) \otimes L_2(\cM) \rightarrow  \ell_2(\sG) \otimes L_2(\cM)   \otimes  L_2(\widehat{\sG})$ as $\widetilde{W} = \Sigma_{23} (W \otimes 1_{B( L_2(\cM) )})$, where $\Sigma_{23}$ flips the second and third tensor coordinate. 	 
	Then for $x \in \cR$ we get that
	\[
	\widetilde{W}^\ast ( x\otimes 1_{B( L_2(\widehat{\sG}) )})  \widetilde{W} = \int_{\gamma \in \widehat{\sG}} W_\gamma^\ast x W_\gamma d\gamma = \int_{\gamma \in \widehat{\sG}} \theta_\gamma( x )  d\gamma = \cE_{\cM}(x).
	\]
	That is, $\widetilde{W}$ is a Stinespring dilation for the conditional expectation $\cE_{\cM}$. We also set $\widetilde{W}^H = \Sigma_{23} (W \otimes 1_{ H })$ as a map $\ell_2(\sG) \otimes H \rightarrow  \ell_2(\sG) \otimes H  \otimes  L_2(\widehat{\sG})$. 	
	Note that 
	\begin{equation}\label{Eqn=House}
	\widetilde{W}^H   \widetilde{V}_t =    (\widetilde{V}_t \otimes 1_{ L_2(\widehat{\sG})  })  \widetilde{W}.
	\end{equation} 
	Further, for $x \in \cR$, 
	\begin{equation}\label{Eqn=CompDetail0}
	\widetilde{\pi}_t 
	\left(   \widetilde{W}^\ast ( x \otimes 1_{B(L_2(\widehat{\sG}))}  ) \widetilde{W}  \right) =
	(\widetilde{W}^H)^\ast ( \widetilde{\pi}_t (x) \otimes 1_{B(L_2(\widehat{\sG}))}  ) \widetilde{W}^H.
	\end{equation}
	
	Now take $x \in \cR^\circ$.
	 As $\cE_{\cM}$ is given by \eqref{Eqn=CEVcross} and $\theta_{\gamma}$ commutes with $\widetilde{\Phi}_t$,  we get that, 
	\begin{equation}\label{Eqn=CompDetail}
	\begin{split}
	&     \Phi_t\left(   \cE_{\cM}(x)^\ast   \cE_{\cM}(x)  \right)  -
	\Phi_t\left( \cE_{\cM}(x) \right)^\ast  \Phi_t\left( \cE_{\cM}(x)  \right) 		\\
	=&     \Phi_t\left( \cE_{\cM}(x)^\ast \cE_{\cM}(x)  \right)  -
	\cE_{\cM}\left( \widetilde{\Phi}_t\left( x \right)^\ast\right)   \cE_{\cM}\left(  \widetilde{\Phi}_t\left( x \right) \right). 
	\end{split}
	\end{equation}
	Now using the Stinespring dilations and \eqref{Eqn=CompDetail0}
	\[
	\begin{split}	 
	\eqref{Eqn=CompDetail} = &  \widetilde{V}_t^\ast  \widetilde{\pi}_t 
\left(   \widetilde{W}^\ast ( x^\ast \otimes 1_{B(L_2(\widehat{\sG}))}  ) \widetilde{W}  \widetilde{W}^\ast (  x  \otimes 1_{B(L_2(  \widehat{\sG} ))}  ) \widetilde{W} \right) \widetilde{V}_t   \\
	& -    \widetilde{W}^\ast (\widetilde{V}_t^\ast   \widetilde{\pi}_t(x)^\ast \widetilde{V}_t \otimes 1_{B( L_2(\widehat{\sG} )) }  )    \widetilde{W}    \widetilde{W}^\ast   ( \widetilde{V}_t^\ast  \widetilde{\pi}_t(x) \widetilde{V}_t  \otimes 1_{ B(L_2(\widehat{\sG} ) )}  )   \widetilde{W}   \\
	= &  \widetilde{V}_t^\ast     (\widetilde{W}^H)^\ast ( \widetilde{\pi}_t (x)^\ast \otimes 1_{B(L_2(\widehat{\sG}))}  ) \widetilde{W}^H  (\widetilde{W}^H)^\ast (  \widetilde{\pi}_t(x)  \otimes 1_{ B( L_2(  \widehat{\sG} ) )}  ) \widetilde{W}^H   \widetilde{V}_t   \\
	& -    \widetilde{W}^\ast (\widetilde{V}_t^\ast   \widetilde{\pi}_t(x)^\ast \widetilde{V}_t \otimes 1_{ B( L_2(\widehat{\sG} ))}  )    \widetilde{W}    \widetilde{W}^\ast   ( \widetilde{V}_t^\ast  \widetilde{\pi}_t(x) \widetilde{V}_t  \otimes 1_{ B( L_2(\widehat{\sG}) )}  )   \widetilde{W}   
	\end{split}
	\]
	Finally we use \eqref{Eqn=House} to find,
	\[
	\begin{split}
		\eqref{Eqn=CompDetail} = & \widetilde{W}^\ast  (\widetilde{V}_t^\ast \widetilde{\pi}_t(x)^\ast (\widetilde{V}_t \widetilde{V}_t^\ast  - 1 )^{\frac{1}{2}}  \otimes 1_{ B( L_2(\widehat{\sG})) }  )   	  
	\widetilde{W}^H \\ & \qquad \times \quad (\widetilde{W}^H)^\ast    
	( (\widetilde{V}_t \widetilde{V}_t^\ast  - 1 )^{\frac{1}{2}} \widetilde{\pi}_t(x) \widetilde{V}_t     \otimes 1_{ B(L_2(\widehat{\sG})) }  )
	\widetilde{W}. 
	\end{split}
	\]
	So that
	\[
	\begin{split}
	&	  \Vert \Phi_t\left( \cE_{\cM}(x)^\ast \cE_{\cM}(x)  \right)  -
	\Phi_t\left( \cE_{\cM}(x) \right)^\ast  \Phi_t\left( \cE_{\cM}(x)  \right) \Vert \\
	\leq &  \Vert     \widetilde{V}_t^\ast \widetilde{\pi}_t(x)^\ast (\widetilde{V}_t \widetilde{V}_t^\ast  - 1 ) \widetilde{\pi}_t(x) \widetilde{V}_t   \otimes 1_{ B(L_2(\widehat{\sG} ))}    \Vert \\
	= &  \Vert     \widetilde{V}_t^\ast \widetilde{\pi}_t(x)^\ast (\widetilde{V}_t \widetilde{V}_t^\ast  - 1 ) \widetilde{\pi}_t(x) \widetilde{V}_t       \Vert 
	=    \Vert \widetilde{\Phi}_t\left( x^\ast x  \right)  -
	\widetilde{\Phi}_t\left( x \right)^\ast  \widetilde{\Phi}_t\left( x \right) \Vert.
	\end{split}
	\] 
	Taking the supremum over $t \geq 0$ we get that $\Vert \cE_\cM(x) \Vert_{\BMO_{\cS}^c} \leq \Vert x \Vert_{\BMO_{\widetilde{\cS}}^c}$. 
	By taking adjoints  we get the row estimate.
	By density of $\cM^\circ$ in $\BMO_{\widetilde{\cS}}(\cR)$ we conclude the proof. 
\end{proof}

\begin{prop}\label{Prop=RnComplement}
	Suppose that $\cS$ admits a $\varphi$-modular standard or reversed Markov-dilation. 
	We have that  $\BMO^\circ_{\widetilde{\mathcal{S}}}(\cR_n)$ is an isometric 1-complemented subspace of $\BMO^\circ_{\widetilde{\mathcal{S}}}(\cR)$.
\end{prop}
\begin{proof}
	The conditional expectation of $\cR$ onto $\cR_n$ is given by  
	\begin{equation}\label{Eqn=CEVRn}
	\cF_{n} (x) =  2^n \int_0^{2^{-n}} \sigma_s^{\widetilde{\varphi}_n}(x) ds, \qquad x \in \cR, 
	\end{equation}
	see \cite{HaagerupJungeXu}. 
	Let $(\cN_t, \pi_t, \cE_t)_{t \geq 0}$ be a $\varphi$-modular Markov dilation for $\cS$ to a von Neumann algebra $\cN$ with normal faithful state $\psi$. By Proposition \ref{Prop=SExtension}  we see that $\widetilde{\cS}$ admits a $\widetilde{\varphi}$-modular Markov dilation $(\cO_t, \widetilde{\pi}_t, \widetilde{\mathcal{E}}_t)_{t \geq 0}$. Moreover, the proof shows that we may take $\cO_t = \cN_t \rtimes_{\sigma^\psi} \sfG$,  $\widetilde{\cE}_t = \cE_t \rtimes_{\sigma^\psi} \sfG$ and $\widetilde{\pi}_t = \pi_t \rtimes \sfG$. Let $k_n$ be the element in $\cO = \cup_{t \geq 0} \cO_t$ that satisfies  $\forall t \geq 0, \widetilde{\pi}_t(h_n) = k_n$   (formally, it is defined as follows: let $d_n = -i \log(\lambda_{2^{-n}}) \in \cO$ (principal branch of the log), then set $e_n = 2^n d_n$ and $k_n = e^{-d_n}$).  Set $\widetilde{\psi}_n = k_n  \widetilde{\psi}    k_n$. 
	The conditional expectation \eqref{Eqn=CEVRn} may be lifted to the $\cO$-level by setting
	\[
	\cF_n^{\cO}(x) = 2^n \int_0^{2^{-n}} \sigma_s^{\widetilde{\psi}_n}(x) ds, \qquad x \in \cO, 
	\]
	We get for $x \in \cR, t \geq 0$ that 
	\begin{equation}\label{Eqn=CommuteE1}
	\begin{split}
	\cF_n^{\cO} \circ \widetilde{\pi}_t(x) = &
	 2^n \int_0^{2^{-n}} \sigma_s^{\widetilde{\psi}_n}( \widetilde{\pi}_t(x)  ) ds 
	 =   2^n \int_0^{2^{-n}} k_n^{it} \sigma_s^{\widetilde{\psi}}( \widetilde{ \pi}_t(x)  )k_n^{-it} ds  \\
	 = &  2^n \int_0^{2^{-n}} \widetilde{ \pi}_t \left( h_n^{it} \sigma_s^{\widetilde{\varphi}}( x  )h_n^{-it}  \right) ds  
	 =   \widetilde{ \pi}_t  \circ  	\cF_n(x).
	 \end{split}
	\end{equation}
	Recall that $\widetilde{\cE}_t$ commutes with $\sigma^{\widetilde{\psi}}_s$ and has $k_n$ in it range so that $\widetilde{\cE}_t(k_n^{is} x k_n^{-is}) = k_n^{is}  \widetilde{\cE}_t(x )k_n^{-is}$. Essentially the same computation as before shows that for  $x \in \cO$ we get that
	\begin{equation}\label{Eqn=CommuteE2}
	\begin{split}
	 \cF_n^{\mathcal{O}}  \circ \widetilde{\cE}_t (x) = &
	2^n \int_0^{2^{-n}} \sigma_s^{\widetilde{\psi}_n}(  \widetilde{\cE}_t(x)  ) ds  
	=   2^n \int_0^{2^{-n}} k_n^{is} \sigma_s^{\widetilde{\psi}}(   \widetilde{\cE}_t(x)  )  )k_n^{-is} ds  \\
	= &  2^n \int_0^{2^{-n}}  \widetilde{\cE}_t\left( k_n^{is} \sigma_s^{\widetilde{\psi}}( x  )k_n^{-is}  \right) ds  
	=   \widetilde{\cE}_t \circ \cF_n^{\cO} (x).
	\end{split}
	\end{equation}
	For $x \in \cR, t \geq 0$ we get using  \eqref{Eqn=MarkovDilationRelation} in the second equality and \eqref{Eqn=CommuteE2} in the third,
\begin{equation} \label{Eqn=ExpectedBMO}
	\begin{split}
      &
	\Vert \widetilde{\Phi}_{t}(\cF_n(x)^\ast \cF_n(x) ) - \widetilde{\Phi}_t(\cF_n(x))^\ast \widetilde{\Phi}_t(\cF_n(x) ) \Vert \\
	= &  \Vert \widetilde{\pi}_{2t}\left( \widetilde{\Phi}_t(\cF_n(x)^\ast \cF_n(x) ) - \widetilde{\Phi}_t(\cF_n(x))^\ast \widetilde{\Phi}_t(\cF_n(x) ) \right) \Vert \\
	= &  \Vert  \widetilde{\cE}_{2t} \left( \widetilde{\pi}_t(\cF_n(x)^\ast \cF_n(x) ) \right) -   \widetilde{\cE}_{2t} ( \widetilde{\pi}_t( \cF_n(x)))^\ast  \widetilde{\cE}_{ 2t}  ( \widetilde{\pi}_t(\cF_n(x) ))   \Vert \\
= &  \Vert  \widetilde{\cE}_{2t} \left(  \cF_n^{\mathcal{O}}(\widetilde{\pi}_t(x)^\ast) \cF_n^{\mathcal{O}}(\widetilde{\pi}_t(x))  \right) - \widetilde{\cE}_{2t} ( \cF_n^{\mathcal{O}}( \widetilde{\pi}_t( x))  )^\ast  \widetilde{\cE}_{ 2t}( \cF_n^{\mathcal{O}}( \widetilde{\pi}_t( x )  ) )  \Vert.
\end{split}
\end{equation}
Next write $P_{2t}$ and $P_n^{\mathcal{O}}$ for the $L_2$-implementation of $\widetilde{\cE}_{2t}$ and  $\cF_{n}^{\mathcal{O}}$ respectively.  Then  $P_{2t}$ and $P_n^{\mathcal{O}}$ are commuting projections and 
$\widetilde{\cE}_{2t}(x) = P_{2t}x P_{2t}$, $\cF_{n}^{\mathcal{O}}(x) = P_n^{\mathcal{O}} x P_n^{\mathcal{O}}$.
Therefore we may estimate \eqref{Eqn=ExpectedBMO} as 
\[
\begin{split}
  &
\Vert \widetilde{\Phi}_{t}(\cF_n(x)^\ast \cF_n(x) ) - \widetilde{\Phi}_t(\cF_n(x))^\ast \widetilde{\Phi}_t(\cF_n(x) ) \Vert \\
=	& \Vert  P_{2t} P_{n}^{\mathcal{O}} \widetilde{\pi}_t( x )^\ast  P_{n}^{\mathcal{O}} \widetilde{\pi}_t( x )    P_{n}^{\mathcal{O}}   P_{2t}  -    P_{2t} P_{n}^{\mathcal{O}} \widetilde{\pi}_t( x )^\ast  P_{n}^{\mathcal{O}}    P_{2t}   P_{n}^{\mathcal{O}}   \widetilde{\pi}_t( x )    P_{n}^{\mathcal{O}}   P_{2t}  \Vert \\
	= &  \Vert P_{n}^{\mathcal{O}}  P_{2t}  \widetilde{\pi}_t( x )^\ast    (1 - P_{2t}) P_{n}^{\mathcal{O}} (1 - P_{2t}) \widetilde{\pi}_t( x )       P_{2t}  P_{n}^{\mathcal{O}}   \Vert \\
	\leq & \Vert   P_{2t}  \widetilde{\pi}_t( x )^\ast    (1 - P_{2t})   \widetilde{\pi}_t( x )       P_{2t}    \Vert. 
	\end{split}
	\]
	By the same computation replacing $\cF_n(x)$ by just $x$ one gets that
	\[ 
	\Vert \widetilde{\Phi}_{t}( x^\ast x ) - \widetilde{\Phi}_t( x )^\ast \widetilde{\Phi}_t( x ) \Vert =  \Vert   P_{2t}  \widetilde{\pi}_t( x )^\ast    (1 - P_{2t})   \widetilde{\pi}_t( x )       P_{2t}    \Vert.
	\]
	So that in all we conclude that 
	\[
	\Vert \widetilde{\Phi}_{t}(\cF_n(x)^\ast \cF_n(x) ) - \widetilde{\Phi}_t(\cF_n(x))^\ast \widetilde{\Phi}_t(\cF_n(x) ) \Vert \leq 
	 \Vert \widetilde{\Phi}_{t}( x^\ast x ) - \widetilde{\Phi}_t( x )^\ast \widetilde{\Phi}_t( x ) \Vert
	\]
	Taking the supremum over all $t \geq 0$ gives $\Vert \cF_n(x) \Vert_{\BMO_{\widetilde{\cS}}^c(\cR_n) } \leq  \Vert x \Vert_{\BMO_{\widetilde{\cS}}^c(\cR) }$, which concludes the proof for the column estimate. The row estimate follows by taking adjoints.  
\end{proof}

Let $\widetilde{\Phi}_t^{(p)}$ and $\widetilde{\Phi}_t^{(p,n)}$ be the semi-groups acting on $L_p(\cR)$ through interpolation with respect to $\widetilde{\varphi}$ and $\widetilde{\varphi}_n$, see \eqref{Eqn=PhiLp}. Note that the definition of the subspace $L_p^\circ(\cR)$ of $L_p(\cR)$ depends on the choice of the state. As we are dealing with different states, namely $\widetilde{\varphi}$ and $\widetilde{\varphi}_n$ these spaces may in principle be different. We distinghuish this in the notation by writing $L_p^\circ(\cR, \widetilde{\varphi})$ and $L_p^\circ(\cR, \widetilde{\varphi}_n)$. The following proposition shows that the spaces are equal however, so that after it we continue writing  $L_p^\circ(\cR)$.

\begin{prop}\label{Prop=LpIso}
Using the notation introduced before Theorem \ref{Thm=Reduction}. Let $1 \leq p < \infty$. We have
\begin{equation}\label{Eqn=Dvergleich}
D_{\varphi}^{\frac{1}{2p}}   =  h^{-\frac{1}{2p}}_n D_{\varphi_n}^{\frac{1}{2p}} =  D_{\varphi_n}^{\frac{1}{2p}} h^{-\frac{1}{2p}}_n.
\end{equation}
Furthermore, we have for $y \in \cR$, 
\begin{equation}\label{Eqn=LpIsoInter}
\kappa_{p}^{\widetilde{\varphi}_n}(y)   =  h^{\frac{1}{2} -\frac{1}{2p}}_n \kappa_{p}^{\widetilde{\varphi}}(y) h^{\frac{1}{2} -\frac{1}{2p}}_n.
\end{equation}
We have $\widetilde{\Phi}_t^{(p,n)} = \widetilde{\Phi}_t^{(p)}$ so that in particular $L_p^\circ(\cR, \widetilde{\varphi}) = L_p^\circ(\cR, \widetilde{\varphi}_n)$. 
The same statements hold if $\cR$ is replaced by $\cR_n$.
\end{prop}
\begin{proof}
	\eqref{Eqn=Dvergleich} is an elementary property of spatial derivatives, see \cite[Section III]{TerpI}. \eqref{Eqn=LpIsoInter}  follows as for $y = D_{\widetilde{\varphi}}^{\frac{1}{2p}}x D_{\widetilde{\varphi}}^{\frac{1}{2p}}, x \in \cR$ we get 
	\[
	\begin{split}
	&  \kappa_{p}^{\widetilde{\varphi}}(D_{\widetilde{\varphi}}^{\frac{1}{2p}} x D_{\widetilde{\varphi}}^{\frac{1}{2p}}  ) = 
D_{\widetilde{\varphi}}^{\frac{1}{2}} x D_{\widetilde{\varphi}}^{\frac{1}{2}}
= h^{-\frac{1}{2}}_n D_{\widetilde{\varphi}_n}^{\frac{1}{2}} x D_{\widetilde{\varphi}_n}^{\frac{1}{2}} h^{-\frac{1}{2}}_n
 = 
 h^{-\frac{1}{2}+\frac{1}{2p}}_n
 \kappa_{p}^{\widetilde{\varphi}_n}(  
 y  ) h^{-\frac{1}{2} + \frac{1}{2p}}_n.
 \end{split}
	\] 
		By using the definitions and Lemma \ref{Lem=MarkovExtension} we see that for $x \in \cR$ we get
	 	 \[	 
	 \begin{split}
		&  \widetilde{\Phi}_t^{(p,n)}(  D_{\widetilde{\varphi}}^{\frac{1}{2p}} x D_{\widetilde{\varphi}}^{\frac{1}{2p}} )
		= \widetilde{\Phi}_t^{(p,n)}(  D_{\widetilde{\varphi}_n}^{\frac{1}{2p}} h_n^{-\frac{1}{2p}} x h_n^{-\frac{1}{2p}} D_{\widetilde{\varphi}_n}^{\frac{1}{2p}} ) 	
		=   D_{\widetilde{\varphi}_n}^{\frac{1}{2p}} \widetilde{\Phi}_t^{(p)}( h_n^{-\frac{1}{2p}} x h_n^{-\frac{1}{2p}}) D_{\widetilde{\varphi}_n}^{\frac{1}{2p}} \\
		= &	D_{\widetilde{\varphi}_n}^{\frac{1}{2p}} h_n^{-\frac{1}{2p}} \widetilde{\Phi}_t^{(p)}(  x ) h_n^{-\frac{1}{2p}} D_{\widetilde{\varphi}_n}^{\frac{1}{2p}} 
		= D_{\widetilde{\varphi}}^{\frac{1}{2p}}  \widetilde{\Phi}_t^{(p)}(  x )  D_{\widetilde{\varphi}}^{\frac{1}{2p}}
		=  \widetilde{\Phi}_t^{(p)}(  D_{ \widetilde{\varphi} }^{\frac{1}{2p}} x D_{ \widetilde{\varphi} }^{\frac{1}{2p}} ).
	 \end{split}
	 \]
	 This shows by density that $\widetilde{\Phi}_t^{(p,n)} = \widetilde{\Phi}_t^{(p)}$. 
\end{proof}

Now consider compatible couples $[\BMO^\circ_{\widetilde{\cS}}(\cR), L_p^\circ(\cR)]^{\widetilde{\varphi}}$ and $[\BMO^\circ_{\widetilde{\cS}}(\cR), L_p^\circ(\cR)]^{\widetilde{\varphi}_n}$ with respect to respective states $\widetilde{\varphi}$ and $\widetilde{\varphi}_n$. Note that $\cR^\circ$ is by definition contained in $\BMO^\circ_{\widetilde{\cS}}(\cR)$.
 Let  $\kappa_{[\BMO,p;q] }^{\widetilde{\varphi}}$  and $\kappa_{[\BMO,p;q] }^{\widetilde{\varphi}_n}$   be the respective natural identifications of $[\BMO^\circ_{\widetilde{\cS}}(\cR), L_p^\circ(\cR)]_{1/q}^{\widetilde{\varphi}}$  and $[\BMO^\circ_{\widetilde{\cS}}(\cR), L_p^\circ(\cR)]_{1/q}^{\widetilde{\varphi}_n}
$  as subspaces of $L_1(\cR)$. 

\begin{prop}\label{Prop=IntIso}
Let $\cS$ be a $\varphi$-modular KMS-symmetric semi-group. We have a complete isometry
\begin{equation}\label{Eqn=InterpolationIso0}
\begin{split}
\sigma_{p,q,n}: & [\BMO^\circ_{\widetilde{\cS}}(\cR), L_p^\circ(\cR)]_{1/q}^{\widetilde{\varphi}} \rightarrow
[\BMO^\circ_{\widetilde{\cS}}(\cR), L_p^\circ(\cR)]_{1/q}^{\widetilde{\varphi}_n}
\end{split}
\end{equation}
Moreover, the isometry is explicitly given by
\[
\kappa_{[\BMO,p;q] }^{\widetilde{\varphi}_n} \circ \sigma_{p,q,n}   (y) 
   =  
    h_n^{\frac{1}{2} - \frac{1}{2pq}}   \kappa_{[\BMO,p;q] }^{\widetilde{\varphi}} (y)  h_n^{\frac{1}{2} - \frac{1}{2pq}}.
\]
\end{prop}
\begin{proof}
We use short hand notation $X = \kappa_{\BMO}^{\widetilde{\varphi}}(\BMO^\circ_{\widetilde{\mathcal{S}}}(\cR))$, $X_n = \kappa_{\BMO}^{\widetilde{\varphi}_n}(\BMO^\circ_{\widetilde{\mathcal{S}}}(\cR))$, $Y = \kappa_p^{\widetilde{\varphi}}(L_p^{\circ}(\cR))$ and $Y_n =  \kappa_p^{\widetilde{\varphi}_n}(L_p^{\circ}(\cR))$.  The norm on  $X$ and $X_n$ is just the norm of $\BMO^\circ_{\widetilde{\mathcal{S}}}(\cR)$ through the respective embeddings  $\kappa_{\BMO}^{\widetilde{\varphi}}$ and $\kappa_{\BMO}^{\widetilde{\varphi}_n}$. Similarly the norms on $Y$ and $Y_n$ is just the norm of  $L_p(\cR)$. Let $\sigma_n$ be the map $[\BMO_{\widetilde{\cS}}^\circ(\cR), L_p^\circ(\cR)]_{\frac{1}{q}}^{\widetilde{\varphi}} \rightarrow L_1^\circ(\cR)$ defined by 
\[
\sigma_n \left(  \kappa^{\widetilde{\varphi}}_{[\BMO, p; q]}(y)\right) = h_n^{\frac{1}{2} - \frac{1}{2pq}}   \kappa_{[\BMO,p;q] }^{\widetilde{\varphi}} (y)  h_n^{\frac{1}{2} - \frac{1}{2pq}},
\]
 i.e. the mapping \eqref{Eqn=InterpolationIso0} on the $L_1$-level.

Take $f \in \mathcal{F}(X, Y)$ which we view as a function on the strip $S \rightarrow X+Y$, where $X+Y$ is a (non-isometric) subspace of $L_1(\cR)$, see Section \ref{Sect=Lp}. Define
\[
(U_n f )(z) =  h_n^{ \frac{iz}{2p} +  \frac{1}{2} } f(z) h_n^{ \frac{iz}{2p} +  \frac{1}{2}   } \in L_1^\circ(\cR), \qquad z \in S.
\]
We claim that $U_n f \in \mathcal{F}(X_n, Y_n)$. 
 Take $s \in \mathbb{R}$ so that by definition of $\mathcal{F}(X, Y)$, $f(s) = \kappa_{\BMO}^{\widetilde{\varphi}}(x)$  for some $x \in \BMO^\circ_{\widetilde{\mathcal{S}}}(\cR)$. Then by \eqref{Eqn=BMOEmbedding} 
\[
(U_n f)(s) = h_n^{ \frac{is}{2p} +  \frac{1}{2}  } f(s) h_n^{  \frac{is}{2p} +  \frac{1}{2} } = \kappa_{\BMO}^{\widetilde{\varphi}_n}( h_n^{ \frac{is}{2p}  } x h_n^{ \frac{is}{2p}  } ).
\]
Further, for $y \in \cR^\circ$ it follows from the definition of the BMO-norm  and Lemma \ref{Lem=MarkovExtension}  that
\[
\begin{split}
 \Vert   h_n^{ \frac{is}{2p}  } y h_n^{ \frac{is}{2p}} \Vert_{  \BMO^c_{\widetilde{\mathcal{S}}}   } ^2 
= & \sup_{t \geq 0}
\Vert
\widetilde{\Phi}_t ( (  h_n^{ \frac{is}{2p}  } y h_n^{ \frac{is}{2p}}  )^\ast (h_n^{ \frac{is}{2p}  } y h_n^{ \frac{is}{2p}}))
-  \widetilde{\Phi}_t (    h_n^{ \frac{is}{2p}  } y h_n^{ \frac{is}{2p}}  )^\ast \widetilde{\Phi}_t(h_n^{ \frac{is}{2p}  } y h_n^{ \frac{is}{2p}}))
\Vert \\
= & \sup_{t \geq 0}
\Vert
h_n^{  -\frac{is}{2p}  }   \widetilde{\Phi}_t (   y^\ast  y ) h_n^{ \frac{is}{2p}}
-  h_n^{ - \frac{is}{2p}  }   \widetilde{\Phi}_t (    y   )^\ast \widetilde{\Phi}_t( y)h_n^{ \frac{is}{2p}  } 
\Vert
= \Vert y \Vert_{  \BMO^c_{\widetilde{\mathcal{S}}}   }^2.
\end{split}
\]
The same holds for the row BMO-space so that $\Vert   h_n^{ \frac{is}{2p}  } y h_n^{ \frac{is}{2p}} \Vert_{  \BMO_{\widetilde{\mathcal{S}}}  } ^2 =  \Vert   y   \Vert_{  \BMO_{\widetilde{\mathcal{S}}}  } ^2$. And by density this in fact holds for all $y \in \BMO_{\widetilde{\mathcal{S}}}^\circ(\cR)$.
This shows that for $s \in \mathbb{R}$ we get $(U_nf)(s) \in X_n$ and
\begin{equation}\label{Eqn=BMOiso}
\Vert (U_n f)(s) \Vert_{\BMO_{\widetilde{\mathcal{S}}}} = \Vert f(s) \Vert_{\BMO_{\widetilde{\mathcal{S}}}}.
\end{equation}

Next consider $i + s \in i + \mathbb{R}$. By definition of $\mathcal{F}(X, Y)$ we have $f(i+s) \in Y$, so write $f(i+s) = \kappa_{p}^{\widetilde{\varphi}}(x)$ for $x \in  L_p^\circ(\cR)$. Then from \eqref{Eqn=Varphi} and \eqref{Eqn=LpIsoInter}
\begin{equation}\label{Eqn=UnIso}
\begin{split}
(U_n f)(i+s) = & h_n^{ \frac{is}{2p} - \frac{1}{2p} + \frac{1}{2}} f(s) h_n^{  \frac{is}{2p}  -  \frac{1}{2p}  + \frac{1}{2}} = h_n^{ \frac{is}{2p}}  \kappa_{p}^{\widetilde{\varphi}_n}(  x )   h_n^{ \frac{is}{2p} } 
=   \kappa_{p}^{\widetilde{\varphi}_n}(  h_n^{  \frac{is}{2p}  }  x  h_n^{  \frac{is}{2p}  } ).
\end{split}
\end{equation}
Proposition  \ref{Prop=LpIso} shows then that $(U_n f)(i+s) \in Y_n$ and
\begin{equation}\label{Eqn=LPiso}
 \Vert (U_n f)(i+s)  \Vert_p  = \Vert f(i+s)\Vert_p.
\end{equation}

We get from the equations \eqref{Eqn=BMOiso} and \eqref{Eqn=LPiso} that   $U_n f \in \mathcal{F}(X_n, Y_n)$ as the fact that $h_n$ is boundedly invertible implies that $U_n f$ is continuous on the strip $S$ and analytic on its interior. Moreover, 
$
\Vert U_n f \Vert_{\mathcal{F}(X_n, Y_n)}
\leq \Vert f \Vert_{\mathcal{F}(X, Y)}$.
So the assignment $f \mapsto U_n f$ is a contraction. 
Consider for $f \in \mathcal{F}(X_n, Y_n)$ the function
\[
(V_n f)(z) =    h_n^{ -\frac{iz}{2p} - \frac{1}{2}}  f(z)   h_n^{ -\frac{iz}{2p} - \frac{1}{2}}, \qquad z \in S.
\]
Then exactly as in the previous paragraph one proves that $V_n f \in \mathcal{F}(X, Y)$ and $
\Vert V_n f \Vert_{\mathcal{F}(X, Y)}
\leq \Vert f \Vert_{\mathcal{F}(X_n, Y_n)}$. Moreover $V_n = U_n^{-1}$ and hence $\mathcal{F}(X, Y)$ and $\mathcal{F}(X_n, Y_n)$ are isometrically isomorphic.

Now take $x \in [X, Y]_{1/q}$. Let $\epsilon > 0$. Take $f \in \mathcal{F}(X, Y)$ such that $f(\frac{i}{q}) = x$ and $\Vert x \Vert_{[X, Y]_{1/q}} \leq \Vert f \Vert_{\mathcal{F}(X, Y)} + \epsilon$. Then $\sigma_n(x) = (U_n f)(\frac{i}{q})$ so that $\Vert \sigma_n(x) \Vert_{[X_n, Y_n]_{1/q}} \leq \Vert U_n f \Vert_{\mathcal{F}(X_n, Y_n)} = \Vert f \Vert_{\mathcal{F}(X, Y)} \leq \Vert x \Vert_{[X,Y]_{1/q}} + \epsilon$. This shows that the map \eqref{Eqn=InterpolationIso0} is well-defined and contractive. Repeating this argument for $V_n$ instead of $U_n$ shows that in fact \eqref{Eqn=InterpolationIso0} is an isometric isomorphism. That the map is completely isometric follows by repeating the argument on matrix levels.
\end{proof}

\begin{thm}\label{Thm=InterpolationBMO}
Let $(\cM, \varphi)$ be a von Neumann algebra with normal faithful state. Let $\mathcal{S} = (\Phi_t)_{t \geq 0}$ be a $\varphi$-modular KMS-symmetric Markov semi-group. Assume that $\mathcal{S}$ admits a reversed Markov dilation with a.u. continuous path. Then we have,  for all $1 \leq p < \infty, 1 < q < \infty$,
\[
[\BMO^\circ_{\mathcal{S}}(\cM), L_p^\circ(\cM)]_{1/q} \approx_{pq} L_{pq}^\circ(\cM).
\]
\end{thm}
\begin{proof}
Because  $\mathcal{S} = (\Phi_t)_{t\geq 0}$ is $\varphi$-modular it may be extended to Markov semi-group $\widetilde{\cS} = (\widetilde{\Phi}_t)_{t \geq 0}$ on $\cR$, see Lemma \ref{Lem=MarkovExtension}. By Proposition \ref{Prop=SExtension} $\widetilde{\cS}$ also has a reversed Markov dilation with  a.u. continuous path. We claim that this map preserves $\cR_n := \cR_{\widetilde{\varphi}_n}$, which was defined as the centralizer of $\widetilde{\varphi}_n$. Let $x \in \cR_n$. Then, by applying \cite[Theorem VIII.3.3]{TakII} twice and Lemma \ref{Lem=MarkovExtension} we get,
\[
 \sigma^{\widetilde{\varphi}_n}_s( \widetilde{\Phi}_t(x)) =  h_n^{is} \sigma^{\widetilde{\varphi}}_s( \widetilde{\Phi}_t(x)) h_n^{-is}
 = h_n^{is} \widetilde{\Phi}_t(  \sigma^{\widetilde{\varphi}}_s( x )  ) h_n^{-is}
 = \widetilde{\Phi}_t(  h_n^{is}  \sigma^{\widetilde{\varphi}}_s( x )   h_n^{-is} )
 = \widetilde{\Phi}_t(x).
\]
So that $x \in \cR_{n}$. Denote the restriction of $\widetilde{\Phi}_{t}$ to $\cR_n$ by   $\widetilde{\Phi}_{n,t}$. In all, we obtained Markov semi-groups $\widetilde{\mathcal{S}} = (\widetilde{\Phi}_{t})_{t \geq 0}$ and $\widetilde{\mathcal{S}}_n = (\widetilde{\Phi}_{n,t})_{t \geq 0}$ with respect to the respective states $\widetilde{\varphi}$ and $\widetilde{\varphi}\vert_{\cR_n}$.
Note that by Lemma \ref{Lem=MarkovExtension}
\[
\widetilde{\varphi}_n \circ \widetilde{\Phi}_t(x) = \widetilde{\varphi}(h_n^{\frac{1}{2}} \widetilde{\Phi}_t(x)  h_n^{\frac{1}{2}} ) = \widetilde{\varphi}(\widetilde{\Phi}_t(   h_n^{\frac{1}{2}} x  h_n^{\frac{1}{2}} ) =  \widetilde{\varphi}(   h_n^{\frac{1}{2}} x  h_n^{\frac{1}{2}}) = \widetilde{\varphi}_n(x).
\]
This shows that $\widetilde{\Phi}_t: \cR_n \rightarrow \cR_n$ is also Markov with respect to  $\widetilde{\varphi}_n$, which is tracial on $\cR_n$.

As the semi-groups $\mathcal{S}$ and $\widetilde{\mathcal{S}}_n$ are restrictions of $\widetilde{\mathcal{S}}$ we have isometric inclusions of the corresponding BMO-spaces
\[
\BMO^\circ_{\mathcal{S}}(\cM)  \subseteq \BMO^\circ_{\widetilde{\mathcal{S}}}(\cR), \qquad  \BMO^\circ_{\widetilde{\mathcal{S}}_n}(\cR_n)  \subseteq \BMO^\circ_{\widetilde{\mathcal{S}}}(\cR), \qquad n \in \mathbb{N}.
\]
Moreover, these inclusions are 1-complemented by Lemmas \ref{Prop=BMOComplementM} and \ref{Prop=RnComplement}. 
Lemma \ref{Lem=MarkovExtension} also shows that  $\widetilde{\cS}$ admits a reversed Markov dilation with a.u. continuous path. Moreover, this dilation may be chosen to be  a dilation with respect to  $\widetilde{\varphi}_n$. Let $m(x) = (m_t(x))_{t \geq 0}$  be the martingale with $x$ in the set $B$ described in Definition \ref{Dfn=AUCtPath} for this Markov dilation. By Lemma \ref{Lem=VanishHp} we see that for every $2\leq p < \infty$ we have $\Vert m(x) \Vert_{h_p^d} = 0$ and then by \eqref{Eqn=HardyExpected} we see that $\Vert m(\cF_n(x)) \Vert_{h_p^d} = 0$. This shows that $\cF_{n}(B)$ is a $\sigma$-weakly  dense subset of $\cR_n$ such that the martingale $m(x), x \in  \cF_{n}(B)$ has vanishing $h_p^d$-norm. Therefore, by Remark \ref{Rmk=HpCondition} the Theorem  \ref{Thm=JungeMei} applies to the  von Neumann algebra $\cR_n$ with normal tracial state $\widetilde{\varphi}_n$ with Markov semi-group $\widetilde{\cS}_n$. 
 
So Theorem \ref{Thm=JungeMei} yields
\[
[\BMO_{\mathcal{S}_n}^\circ(\cR_n), L_q^\circ(\cR_n)  ]_{1/p}^{\widetilde{\varphi}_n} \approx_{pq} L_{pq}(\cR_n)^\circ.
\]
Now we have isometries
\[
 \xymatrix{
  [\BMO^\circ_{\widetilde{\mathcal{S}}_n}(\cR_n), L_{p}^\circ( \cR_n)   ]_{1/q}^{\widetilde{\varphi}_n}    \qquad \ar@{->}[r]^{ {\rm Lemma } \: \ref{Lem=OneCompl} }      &   [\BMO^\circ_{\widetilde{\mathcal{S}}}(\cR), L_{p}^\circ( \cR)  ]_{1/q}^{\widetilde{\varphi}_n}   \ar@{->}[d]^{\sigma_{p,q,n}^{-1} \:\:\:  \:\: {\rm Prop. } \: \ref{Prop=IntIso} }   \\
   L_{pq}^\circ(\cR_n ) \ar@{->}[u]^{\:\:\:\:  \approx_{pq}   }   &   [\BMO^\circ_{\cS}(\cR), L_p^\circ(\cR)]_{1/q}^{\widetilde{\varphi}}.
}
\]
Furthermore, for $x \in \cR_n$,
\[ 
\begin{split}
 & \kappa_{[\BMO, p; q]}^{\widetilde{\varphi}} \circ  \sigma_{p,q,n}^{-1}(D_{\widetilde{\varphi}}^{\frac{1}{2pq}} x  D_{\widetilde{\varphi}}^{\frac{1}{2pq}}) 
 =   h_n^{-\frac{1}{2} + \frac{1}{2pq}} \kappa_{[\BMO, p; q]}^{\widetilde{\varphi}_n}  (D_{\widetilde{\varphi}}^{\frac{1}{2pq}} x  D_{\widetilde{\varphi}}^{\frac{1}{2pq}}) h_n^{-\frac{1}{2} + \frac{1}{2pq}}  \\
 = &  h_n^{-\frac{1}{2} + \frac{1}{2pq}}  \kappa_{[\BMO, p; q]}^{\widetilde{\varphi}_n}  (D_{\widetilde{\varphi_n}}^{\frac{1}{2pq}} h_n^{-  \frac{1}{2pq}} x h_n^{-\frac{1}{2pq}} D_{\widetilde{\varphi_n}}^{\frac{1}{2pq}}) h_n^{-\frac{1}{2} + \frac{1}{2pq}} \\
 = & h_n^{-\frac{1}{2} + \frac{1}{2pq}}   D_{\widetilde{\varphi_n}}^{\frac{1}{2}} h_n^{-  \frac{1}{2pq}} x h_n^{-\frac{1}{2pq}} D_{\widetilde{\varphi_n}}^{\frac{1}{2}}) h_n^{-\frac{1}{2} + \frac{1}{2pq}} \\
 = & D_{\widetilde{\varphi}}^{\frac{1}{2}} x  D_{\widetilde{\varphi}}^{\frac{1}{2}}.
 \end{split} 
\] 
It follows that for each $n \in \mathbb{N}$ we have an isometric embedding,
\[
j_n: L_{pq}^\circ(\cR_n) \rightarrow [\BMO_{\widetilde{S}}^\circ(\cR), L_p^\circ(\cR)]_{1/q}^{\widetilde{\varphi}}, 
\]
and these embeddings are compatible with the inclusions $L_{pq}^\circ(\cR_n) \subseteq L_{pq}^\circ(\cR_{n+1})$ with respect to $\widetilde{\varphi}$.  This shows that $\cup_{n \in \mathbb{N}} L_{pq}^\circ(\cR_n, \widetilde{\varphi})$ can  isometrically be identified with a subspace of    $[\BMO_{\widetilde{S}}(\cR), L_p^\circ(\cR )]_{1/q}^{\widetilde{\varphi}}$. As $\cup_{n \in \mathbb{N}} L_{pq}^\circ(\cR_n)$ is dense in $L_{pq}^\circ(\cR)$, c.f. \cite[Theorem 8]{Goldstein}, we see that $L_{pq}^\circ(\cR)$ is isometrically contained in the space $[\BMO_{\widetilde{\mathcal{S}}}(\cR), L_p^\circ(\cR )]_{1/q}^{\widetilde{\varphi}}$.  By \cite[Theorem 4.2.2.(a)]{BerghLofstrom} we have that $\cR^{\circ}$ is dense in $[\BMO_{\widetilde{\mathcal{S}}}^\circ(\cR), L_q^\circ(\cR)]_{1/p}^{\widetilde{\varphi}}$. Further as  $\cR^{\circ}$ is also contained in $L_{pq}^\circ(\cR)$ we must have an    isomorphism
  \begin{equation}\label{Eqn=RInterPol}
  [\BMO_{\mathcal{S}_n}^\circ(\cR), L_q^\circ(\cR)]_{1/p} \approx_{pq} L_{pq}^\circ(\cR).
  \end{equation}

  Now again by Lemma \ref{Lem=OneCompl} we see that the space $[\BMO^\circ_{\mathcal{S}}(\cM), L^\circ_p(\cM)]_{1/q}^{\varphi}$ is a 1-complemented subspace of the left hand side of  \eqref{Eqn=RInterPol} and hence of  $L_{pq}^\circ(\cR)$. Further by \cite[Theorem 4.2.2.(a)]{BerghLofstrom} the space $[\BMO_{\mathcal{S}_n}^\circ(\cM), L_p^\circ(\cM)]_{1/q}^{\varphi}$  contains $\cM^\circ$ densely. Since in turn $\cM^\circ$ is dense in $L_{pq}^\circ(\cM)$ which is included in $L_{pq}^\circ(\cR)$ isometrically,  we conclude that  $[\BMO_{\mathcal{S}_n}^\circ(\cM), L_q(\cM)^\circ]_{1/p}^\varphi \approx_{pq} L_{pq}^\circ(\cM)$. Isomorphisms holds for complete bounds by considering matrix levels.
 \end{proof}

\section{Other  BMO-spaces associated with Markov semi-groups}\label{Sect=OtherBMO}

As in the rest of this paper let $\cM$ be von Neumann algebra with faithful normal state $\varphi$. Let $\cS = (\Phi_t)_{t \geq 0}$ be a Markov semi-group on $\cM$. Define a semi-norm (see \cite[Proposition 2.1]{JungeMei}),
\[
\Vert x \Vert_{\BMOa_{\cS}^c} = 
\sup_{t \geq 0}  \Vert  \Phi_t( \vert x - \Phi_t(x) \vert^2) \Vert^{\frac{1}{2}},
\]
and then set
\[
\Vert x \Vert_{\BMOa_{\cS}^r} = \Vert x^\ast \Vert_{\BMOa_{\cS}^c}, \qquad \Vert x \Vert_{\BMOa_{\cS}} = \max(\Vert x \Vert_{\BMOa_{\cS}^c}, \Vert x \Vert_{\BMOa_{\cS}^r}).
\]
\begin{lem}
For $x \in \cM$ we have $\Vert x \Vert_{\BMOa_{\cS}} =0$ if and only if for all $t \geq 0, \Phi_t(x) = x$. 
\end{lem}
\begin{proof}
	 The if part is obvious. Conversly, if $\Vert x \Vert_{\BMOa_{\cS}} =0$ then for all $t \geq 0$ we have $\Vert  \Phi_t( \vert x - \Phi_t(x) \vert^2) \Vert = 0$ and so $0 = \varphi(\Phi_t( \vert x - \Phi_t(x) \vert^2) ) = \varphi(  \vert x - \Phi_t(x) \vert^2)$. As $\varphi$ is faithful $x = \Phi_t(x)$.
\end{proof}
We see that on $\cM^\circ$ the $\BMOa$-semi-norm is actually a norm and its completion will be denoted by $\BMOa_{\cS}^\circ$ or $\BMOa_{\cS}^\circ(\cM)$. Note that we do not need to assume KMS-symmetry here. 

Furthermore, let $A_2$ be the closed densely defined operator such that $\exp(-tA_2) = \Phi_t^{(2)}, t \geq 0$, see Section \ref{Sect=Semigroup}. The Poisson semi-group $\cP = (\Psi_t)_{t\geq 0}$ is defined as the unique Markov semi-group such that $\Psi_t^{(2)} = \exp(-t A_2^{\frac{1}{2}}), t\geq 0$ (see \cite{Sauvageot}).  Therefore we obtain BMO-spaces
\[
\BMO_{\cP}^\circ = \BMO_{\cP}^\circ(\cM), \qquad 
\BMOa_{\cP}^\circ = \BMOa_{\cP}^\circ(\cM),
\]
together with their obvious row and column counterparts. Then \cite[Theorem 5.2]{JungeMei} proves the following tracial interpolation result.

\begin{thm}\label{Thm=TracialOtherBMO}
 	Let $\cM$ be a von Neumann algebra with faithful normal tracial state $\varphi$. Let $\cS = (\Phi_t)_{t \geq 0}$ be a KMS-symmetric Markov semi-group for $(\cM, \varphi)$.
	  Assume that $\cS$ admits a standard Markov dilation. Then,  
  \[
  [X, L_p^\circ(\cM)]_{1/q} \approx_{pq} L_{pq}^\circ(\cM),
  \] 
  where $X$ is any of the spaces $\BMOa^\circ_{\mathcal{S}}(\cM), \BMO^\circ_{\cP}(\cM)$ or $\BMOa^\circ_{\cP}(\cM)$.
\end{thm}

We may generalize this to the non-tracial setting in the following way. The proof follows closely the lines of   Theorem \ref{Thm=InterpolationBMO}. We give the main differences. Firstly, we have that  $\BMOa^\circ_{\mathcal{S}}(\cM)$ embeds contractively into $L_1(\cM)$ as for $x \in \cM^\circ$ with polar decomposition $x = u \vert x \vert$ we get that
\[
\begin{split}
\Vert D_{\varphi}^{\frac{1}{2}} x D_{\varphi}^{\frac{1}{2}} \Vert_1^2  = &
\Vert D_{\varphi}^{\frac{1}{2}} u \vert x \vert D_{\varphi}^{\frac{1}{2}} \Vert_1^2
\leq \Vert D_{\varphi}^{\frac{1}{2}} u \Vert_2^2 \Vert \vert x \vert D_{\varphi}^{\frac{1}{2}} \Vert_2^2 
\leq \Vert D_{\varphi}^{\frac{1}{2}} x^\ast x D_{\varphi}^{\frac{1}{2}} \Vert_1
= \varphi(x^\ast x)\\
 = & \lim_{t \rightarrow \infty} \varphi(x^\ast x  + \Phi_t(x)^\ast \Phi_t(x)  - 
 \Phi_t(x)^\ast x - x^\ast \Phi_t(x)  ) \\
 = & \lim_{t \rightarrow \infty} \varphi( \Phi_t( x^\ast x  + \Phi_t(x)^\ast \Phi_t(x)  - 
 \Phi_t(x)^\ast x - x^\ast \Phi_t(x) ) ) \\
 \leq  & \limsup_{t \rightarrow \infty} \varphi( \Phi_t( x^\ast x  + \Phi_t(x)^\ast \Phi_t(x)  - 
 \Phi_t(x)^\ast x - x^\ast \Phi_t(x) ) ) 
 \leq  \Vert x \Vert_{\BMOa^\circ_{\mathcal{S}}}^2.
\end{split}
\] 
A similar argument holds for the row estimate, which yields a version of Lemma \ref{Lem=BMOinL1} for $\BMOa^\circ_{\mathcal{S}}$. Similarly the spaces   $\BMOa^\circ_{\mathcal{P}}$ and $\BMO^\circ_{\mathcal{P}}$ embed contractively into $L_1(\cM)$. The same statements hold for the completely bounded norms by considerig matrix amplifications. Let $X$ be any of these spaces. We denote the embedding of the complex interpolation spaces by 
\[
\kappa_{[X,p ; q]}^{\varphi}: [X, L_p^\circ(\cM)]_{1/q}^{\varphi} \rightarrow L_1^\circ(\cM). 
\]

\begin{lem}
	Let $\cM_1$ be a von Neumann subalgebra of $\cM$ that is invariant under the semi-group $\cS$ and which admits a $\varphi$-preserving conditional expectation $\cE$. Then we have 1-complemented inclusions
	\begin{equation}\label{Eqn=Inc}
	  \BMOa^\circ_{\mathcal{S}}(\cM_1) \subseteq \BMOa^\circ_{\mathcal{S}}(\cM), \quad
	  \BMOa^\circ_{\mathcal{P}}(\cM_1) \subseteq \BMOa^\circ_{\mathcal{P}}(\cM).
	\end{equation}
	Moreover, we have a 1-complemented inclusion $\BMO^\circ_{\mathcal{P}}(\cM_1) \subseteq \BMO^\circ_{\mathcal{P}}(\cM)$ and if $\cS$ admits a   standard Markov dilation we have a 1-complemented inclusion $\BMO^\circ_{\widetilde{\mathcal{P}}}(\cR_n) \subseteq \BMO^\circ_{\widetilde{\mathcal{P}}}(\cR)$.
\end{lem}
\begin{proof}
	It is immediate that \eqref{Eqn=Inc} are isometric inclusions.  Also for any $t \geq 0$ by the Kadison-Schwarz inequality,
	\[
	\begin{split}
	  & \Vert \Phi_t ( \vert  \cE(x) - \Phi_t(\cE(x)) \vert^2 ) \Vert^2
	=   \Vert \Phi_t (   \cE(  x - \Phi_t(x) )^\ast \cE(  x - \Phi_t(x) ) )  \Vert^2 \\
	\leq & \Vert   \Phi_t (  \cE((  x - \Phi_t(x) )^\ast (  x - \Phi_t(x) ) ) ) \Vert^2	
	\leq \Vert   \Phi_t (   (  x - \Phi_t(x) )^\ast (  x - \Phi_t(x) ) )  \Vert^2.
	\end{split}
	\]
	Taking the supremum over $t \geq 0$ we see that $\Vert \cE(x) \Vert_{\BMOa^\circ_{\mathcal{S}}(\cM_1)} \leq \Vert x \Vert_{\BMOa^\circ_{\mathcal{S}}(\cM)}$. The same argument applies to the Poisson semi-group $\cP$ so that \eqref{Eqn=Inc} follows. Accoding to \cite{Anan} a standard Markov dilation for $\cS$ yields a Markov dilation for $\mathcal{P}$.  The proof of the remaining statements are then similar to Lemmas \ref{Prop=BMOComplementM} and \ref{Prop=RnComplement}. 
\end{proof}

The proof of the following proposition is similar to the one of Proposition \ref{Eqn=InterpolationIso}.

\begin{prop}\label{Prop=IntIso}
	Let $\cS$ be a $\varphi$-modular semi-group. Let $X$ be any of the spaces $\BMOa^\circ_{\widetilde{\mathcal{S}}}, \BMO^\circ_{\widetilde{\cP}}$ or $\BMOa^\circ_{\widetilde{\cP}}$. We have a complete isometry
	\begin{equation}\label{Eqn=InterpolationIso}
	\begin{split}
	\sigma_{X,p,q,n}: & [X, L_p^\circ(\cR)]_{1/q}^{\widetilde{\varphi}} \rightarrow
	[X, L_p^\circ(\cR)]_{1/q}^{\widetilde{\varphi}_n}
	\end{split}
	\end{equation}
	Moreover, the isometry is explicitly given by 
	\[	
	\kappa_{[X,p;q] }^{\widetilde{\varphi}_n} \circ \sigma_{X,p,q,n}   (y) 
	=  
	h_n^{\frac{1}{2} - \frac{1}{2pq}}   \kappa_{[X,p;q] }^{\widetilde{\varphi}} (y)  h_n^{\frac{1}{2} - \frac{1}{2pq}}.
	\]	
\end{prop}

We now get the following theorem. The KMS-symmetry is only needed because Theorem \ref{Thm=TracialOtherBMO} assumes it. 

\begin{thm}
	Following the notation introduced above. Assume moreover that $\cS$ is a $\varphi$-modular KMS-symmetric Markov semi-group that admits a standard Markov dilation. Then,    for all $1 \leq p < \infty, 1 < q < \infty$,
	\[
	[X, L_p^\circ(\cM)]_{1/q} \approx_{pq} L_{pq}^\circ(\cM),
	\] 
	where $X$ is any of the spaces $\BMOa^\circ_{\mathcal{S}}, \BMO^\circ_{\cP}$ or $\BMOa^\circ_{\cP}$.
\end{thm}
\begin{proof}
	We sketch the proof. First we observe that again $\cS$ may be extended to a Markov semi-group $\widetilde{\cS}$ on $\cR$ which has a standard Markov dilation, c.f. Proposition \ref{Prop=SExtension}. Again $\widetilde{\cS}$ restricts to $\cR_n$ as a Markov semi-group with respect to $\widetilde{\varphi}_n$.  Depending on which space $X$ is (as in the statement of the theorem) we define the following. Let $Y$ be either $\BMOa^\circ_{\mathcal{S}}(\mathcal{R}), \BMO^\circ_{\cP}(\mathcal{R})$ or $\BMOa^\circ_{\cP}(\mathcal{R})$. Let $Y_n$ be either $\BMOa^\circ_{\mathcal{S}}(\mathcal{R}_n), \BMO^\circ_{\cP}(\mathcal{R}_n)$ or $\BMOa^\circ_{\cP}(\mathcal{R}_n)$.  We may therefore apply the tracial Theorem \ref{Thm=TracialOtherBMO} to interpolate for each $n$ and find $[Y_n, L_p^\circ(\cR_n)]_{1/q}^{\widetilde{\varphi_n}} \approx_{pq} L_{pq}^\circ(\cR_n)$. One now checks that there is a  diagram   
	\[
	\xymatrix{
		[ Y_n, L_{p}^\circ( \cR_n)   ]_{1/q}^{\widetilde{\varphi}_n}    \qquad \ar@{->}[r]^{\subseteq}     &   [ Y, L_{p}^\circ( \cR)  ]_{1/q}^{\widetilde{\varphi}_n}   \ar@{->}[d]^{\sigma^{-1}_{X,p,q,n}}     \\
		L_{pq}^\circ(\cR_n ) \ar@{->}[u]^{\:\:\:\:  \approx_{pq}   }   &   [ Y, L_p^\circ(\cR)]_{1/q}^{\widetilde{\varphi}}. 
	}
	\]
 	that is compatible with respect to the interpolation structure of $\widetilde{\varphi}$. The remainder of the proof is then exactly the same as in Theorem \ref{Thm=InterpolationBMO}. 
\end{proof}

 \section{Fourier multipliers on free Araki-Woods factors}\label{Sect=Hilbert}

We recall the definiton of free Araki-Woods factors from \cite{Shlyakhtenko}. Let $\cHR$ be a real Hilbert space and let $\cHC = \cHR \otimes_{\mathbb{R}} \mathbb{C}$ be its complexification. 
For $\xi \in \cHC$ with $\xi = \xi_1 + i \xi_2$ and $\xi_1, \xi_2 \in \cHR$ we set $\overline{\xi} = \xi_1 - i \xi_2$. 
 Let $(V_t)_{t \in \mathbb{R}}$ be a strongly continuous 1-parameter group of orthogonal transformations on $\cHR$ and use the same notation for its extension to a strongly continuous unitary 1-parameter group on $\cHC$.  Through Stone's theorem we have $V_t = A^{it}$ where $A$ is a positive (possibly) unbounded self-adjoint operator on $\cHC$. We define a new innerproduct on $\cHC$ by setting $\langle \xi, \eta \rangle_A = \langle \frac{2A}{1+A} \xi, \eta \rangle$. Let $\cH$ be the completion of $\cHC$ with respect to the latter inner product. We have that the embedding $\cHR \hookrightarrow \cH$ is isometric \cite[p. 332]{Shlyakhtenko}. 
   We construct a Fock space,
\[
\cF = \mathbb{C} \Omega \oplus \bigoplus_{k=1}^{\infty} \cH^{\otimes k}.
\]
We denote $\varphi_{\Omega}$ for the vector state $x \mapsto \langle x \Omega, \Omega \rangle$. 
For $\xi \in \cH$ let $a(\xi)$ be the creation operator on $\mathcal{F}$ defined by
\[
a(\xi): \eta_1 \otimes \ldots \otimes \eta_k \mapsto \xi \otimes \eta_1 \otimes \ldots \otimes \eta_k.
\]
Let $a^\ast(\xi)$ be its adjoint which is the annihilation operator
\[
a^\ast(\xi): \eta_1 \otimes \ldots \otimes \eta_k \mapsto \langle \eta_1, \xi \rangle_A \eta_2 \otimes \ldots \otimes \eta_k.
\]
For $\xi \in \cH$ define the self-adjoint operator $s(\xi) = a(\xi) + a^{\ast}(\xi)$. Let,
\[
\cM := \mathcal{A}_0'' \textrm{ with }  \mathcal{A}_0 := \Gamma(\cHR, (V_t)_t) := \langle s(\xi) \mid  \xi \in \cHR   \rangle,
\]
where $\langle s(\xi) \mid  \xi \in \cHR   \rangle$ stands for the $\ast$-algebra generated by these operators. 
The von Neumann algebra $\cM$ is called the {\it free Araki-Woods} algebra. The vacu\"um vector $\Omega$ is separating and cyclic for this algebra. Set $\varphi_\Omega(\: \cdot \: ) = \langle \: \cdot \:  \Omega, \Omega \rangle$. Therefore {\it if} for $\xi \in \cF$ there is an operator $W(\xi)$ such that $W(\xi) \Omega = \xi$, then this operator is unique.  For various calculations and to define suitable Fourier multipliers in the first place we need the following Wick theorem.
\begin{thm}[See Proposition 2.7 of \cite{BKS} or Lemma 3.2 of \cite{HoudayerRicard}]\label{Thm=Wick}
Suppose that $\xi_1, \ldots, \xi_n \in \cHC$ then,
\begin{equation}\label{Eqn=Wick}
W(\xi_1 \otimes \ldots \otimes \xi_n) = \sum_{j=0}^n  a(\xi_1) \ldots a(\xi_j) a^\ast(\overline{ \xi_{j+1} } ) \ldots a^\ast( \overline{ \xi_{n} }).
\end{equation}
\end{thm}
 The linear span of operators of the form \eqref{Eqn=Wick} form a $\ast$-algebra which we shall denote by $\mathcal{A}$ (in fact, this follows from \eqref{Eqn=Split} below). Moreover $\mathcal{A}$ is  dense in $\cM$. 
	
If $T$ is a contractive operator on $\cH_{\mathbb{R}}$ such that for every $t \in \mathbb{R}$ we have $T V_t = V_t T$  then there exists a unique normal ucp map (see \cite{Hiai}, \cite[Proposition 3.3]{Wasilewski} for the even more general result for the $q$-Araki-Woods case),
\[
\Gamma(T): \cM \rightarrow \cM: W(\xi_1 \otimes \ldots \otimes \xi_n) \mapsto W(T\xi_1 \otimes \ldots \otimes T\xi_n).
\]
This assignment is called {\it second quantization}. We are now ready to define the Hilbert transform. 
 
\begin{dfn}\label{Dfn=FreeHilbert}
Fix spaces $\cH_\mathbb{C}^{\pm} \subseteq \cHC$ that are closed in $\cHC$ and such that $\cH_\mathbb{C}^{+} \cap \cH_\mathbb{C}^{-} = \{0\}$. So as Banach spaces $\cHC = \cHC^+ \oplus \cHC^-$. 
Assume moreover that  $\cH_\mathbb{C}^{+}$ and  $\cH_\mathbb{C}^{-}$ are orthogonal in $\cH$ for the inner product $\langle \: , \: \rangle_A$. Set $\epsilon = (\cH_\mathbb{C}^{+}, \cH_\mathbb{C}^{-})$.  The mapping $H_\epsilon: \mathcal{A} \rightarrow \mathcal{A}$ defined as the linear extension of
\[
H_\epsilon: W(\xi_1 \otimes \ldots \otimes \xi_n) =  \pm W( \xi_1 \otimes \ldots \otimes \xi_n),
\]
with $\xi_{1} \in \cH_\mathbb{C}^{\pm}, \xi_2, \ldots, \xi_n \in \cHC$
and $H_\epsilon(1) = 0$ will be called the {\it Hilbert transform} (which only depends on the decomposition $\cHC = \cHC^+ \oplus \cHC^-$).
\end{dfn}

\begin{rmk}\label{Rmk=GaussianModular}
Let $(\sigma_t)_{t \in \mathbb{R}}$ be the modular automorphism group of $\varphi_\Omega$. We have 
\begin{equation}\label{Eqn=GaussianModular}
\sigma_t(W(\xi_1 \otimes \ldots \otimes \xi_n)) = W( (A^{it} \xi_1) \otimes \ldots \otimes (A^{it} \xi_n) ),
\end{equation}
 see \cite{Shlyakhtenko}. Suppose that the spaces $\cHC^{\pm}$ are invariant subspaces for all $A^{it}, t \in \mathbb{R}$.  It follows that $\sigma_t$ and $H_\epsilon$ with $\epsilon = (\cH_\mathbb{C}^{+}, \cH_\mathbb{C}^{-})$ commute for all $t \in \mathbb{R}$.
\end{rmk}

\subsection{$L^p$-boundedness and Cotlar's trick}

To define a fixed Hilbert transform we  prefix a decomposition $\cHC^+ \oplus \cHC^-$ of the Hilbert space $\cHC$. Here  the $\cHC^{\pm}$ are closed in $\cHC$ and orthogonal in $\cH$.  Set $\epsilon = (\cH_\mathbb{C}^{+}, \cH_\mathbb{C}^{-})$ as before.
We write $\cE_\Omega^\perp(x) = x - \varphi_{\Omega}(x)$. This is the orthocomplement of the projection onto $\mathbb{C} 1 \subseteq \cM$ with respect to the inner product of the vacu\"um state.

\begin{prop}[Cotlar formula for the Hilbert transform]
The following relation holds true:
\begin{equation}\label{Eqn=Cotlar}
\cE_\Omega^\perp\left( H_{\epsilon}(x) H_{\epsilon}(y)^\ast  \right) = \cE_\Omega^\perp\left(   H_{\epsilon}( x H_{\epsilon}(y)^\ast) + H_{\epsilon}( y H_{\epsilon}(x)^\ast)^\ast -  H_{\epsilon}(H_{\epsilon}(xy^\ast)^\ast)^\ast \right),
\end{equation}
for all $x,y \in \mathcal{A}$.
\end{prop}
\begin{proof}
By linearity we may assume that $x$ and $y$ are Wick operators of elementary tensors. So say
$x = W(  \xi_1  \otimes \ldots \otimes  \xi_m )$ and $y = W(  \eta_1  \otimes \ldots \otimes  \eta_n )$. Moreover, assume that $\xi_1 \in \cHC^{\epsilon_x}, \eta_1 \in \cHC^{\epsilon_y}$ for signs $\epsilon_x, \epsilon_y = \pm 1$.
By \eqref{Eqn=Wick} we get,
\begin{equation}\label{Eqn=xystar}
\begin{split}
&x y^\ast =  \\
&
\left(   \sum_{r=0}^m  a( \xi_1 )   \ldots   a(  \xi_r )    a^\ast( \overline{ \xi_{r+1} }  )      \ldots   a^\ast(\overline{ \xi_m }  )   \right)
\left(       \sum_{s=0}^n  a( \overline{ \eta_n } )  \ldots   a( \overline{ \eta_{s+1} }  )
a^\ast( \eta_s  )    \ldots  a^\ast( \eta_1 ).
 \right)
\end{split}
\end{equation}
We rename vectors by setting $(\mu_{1}, \ldots, \mu_{n+m}) = ( \xi_{1} , \ldots,  \xi_{m} ,  \overline{ \eta_n }, \ldots, \overline{ \eta_1 })$.
In the first equality in the next computation we collect the terms in \eqref{Eqn=xystar} by separating the ones where no annihilation operator is on the left of a creation operator (first summand of \eqref{Eqn=Split}) and the ones where such a combination does occur (second summand of \eqref{Eqn=Split}). The second equation is the Wick formula \eqref{Eqn=Wick},
\begin{equation}\label{Eqn=Split}
\begin{split}
x y^\ast =   &     \sum_{r=0}^{n+m} a(\mu_1)   \ldots   a(\mu_r)    a^\ast(\overline{\mu_{r+1}})      \ldots   a^\ast(\overline{\mu_{n+m}})   \\
  & +  \left(   \sum_{r=0}^{m-1}      a( \xi_1 )  \ldots a( \xi_r )    a^\ast(\overline{\xi_{r+1}} )      \ldots   a^\ast(\overline{ \xi_{m-1}})   \right)  a^\ast(\overline{ \xi_{m}}) a(  \overline{ \eta_{n}} )
       \\
      & \quad \times \quad \left(   \sum_{s=0}^{n-1}      a(\overline{\eta_{n-1})}  \ldots   a(\overline{ \eta_{s+1} })    a^\ast(  \eta_{s}  )    \ldots  a^\ast(\eta_{1})\right) \\
 = & W(  \xi_{1}  \otimes \ldots \otimes  \xi_{m} \otimes \overline{ \eta_{n}} \otimes \ldots \otimes \overline{ \eta_{1} }   ) \\
  & +    \langle \overline{ \eta_{n} },  \overline{ \xi_{m} }  \rangle_A W( \xi_{1}  \otimes \ldots \otimes   \xi_{m-1}  ) W( \eta_{1} \otimes \ldots \otimes \eta_{n-1}  )^\ast.
\end{split}
\end{equation}
Now we separate cases. Note that we may assume that $n, m \not = 0$ because otherwise the proposition is trivial. 

\vspace{0.3cm}
 
\noindent {\it Case 1:} Assume $m > n >0$.   Applying the equation \eqref{Eqn=Split} inductively on the length of $m$ we see that,
\begin{equation}\label{Eqn=xystar2}
\begin{split}
x y^\ast = & \sum_{k=0}^n \prod_{l=0 }^{k-1} \langle \overline{ \eta_{n-l} }  , \overline{ \xi_{m-l}}  \rangle_A  W( \xi_{1}  \otimes \ldots \otimes  \xi_{m-k}  \otimes  \overline{ \eta_{n-k} } \otimes \ldots \otimes \overline{ \eta_{1}} ).
\end{split}
\end{equation}
Here the $k=0$ term is understood as
 $W( \xi_{1}  \otimes \ldots \otimes  \xi_{m}  \otimes  \overline{ \eta_{n} } \otimes \ldots \otimes \overline{ \eta_{1}} )$.
In particular as $m > n$ we see that $\cE_\Omega^\perp\left( x y^\ast \right) =  xy^\ast$ and similarly
\[
\begin{split}
\cE_\Omega^\perp\left( H_{\epsilon}( x H_{\epsilon}(y)^\ast)  \right) = & H_{\epsilon}( x H_{\epsilon}(y)^\ast),\\ \cE_\Omega^\perp\left(H_{\epsilon}( y H_{\epsilon}(x)^\ast)^\ast \right) = & H_{\epsilon}( y H_{\epsilon}(x)^\ast)^\ast, \\ 
 \cE_\Omega^\perp\left(  H_{\epsilon}(H_{\epsilon}(xy^\ast)^\ast)^\ast \right) = & H_{\epsilon}(H_{\epsilon}(xy^\ast)^\ast)^\ast.
 \end{split}
 \]
 So   to prove the Cotlar identity \eqref{Eqn=Cotlar} we can ignore the projection  $\cE_\Omega^\perp$ in this case. 
Now for the right hand side of the Cotlar identity \eqref{Eqn=Cotlar} we argue that we get the  Equation \eqref{Eqn=RHSCotlar} below.   Firstly, because $H_{\epsilon}(y) = \epsilon_{y} y$ we find that  $x H_{\epsilon}(y)^\ast$ equals $ \epsilon_{y}$ times the expression \eqref{Eqn=xystar2}. Then, as $m > n$,
\[
\begin{split}
H_{\epsilon}( x H_{\epsilon}(y)^\ast)
= & \sum_{k=0}^{n}   \epsilon_{y} \epsilon_{x}     \prod_{l=0 }^{k-1} \langle  \overline{ \eta_{n-l} }   , \overline{\xi_{m-l}}  \rangle_A  W( \xi_{1}  \otimes \ldots \otimes  \xi_{m-k}  \otimes  \overline{ \eta_{n-k} } \otimes \ldots \otimes \overline{ \eta_{1} } ).
\end{split}
\] 
Secondly, $y H_{\epsilon}(x)^\ast = \epsilon_{x} y  x^\ast$. Then, $y H_{\epsilon}(x)^\ast$ is $ \epsilon_{x}$ times the adjoint of the expression \eqref{Eqn=xystar2}. Then, we get the following two summands, where the second line appears as if $k =n$ then there is no more tensor $\eta_{1}$ appearing  in the decomposition of $y H_{\epsilon}(x)^\ast$ in terms of Wick words,
\[
\begin{split}
H_{\epsilon}( y H_{\epsilon}(x)^\ast)= &
 \sum_{k=0}^{n-1}     \epsilon_{x} \epsilon_{y}     \prod_{l=0 }^{k-1} \langle  \overline{ \xi_{m-l} }   , \overline{ \eta_{n-l} }  \rangle_A  W(   \eta_{1} \otimes \ldots \otimes    \eta_{n-k}  \otimes  \overline{  \xi_{m-k  } } \otimes \ldots \otimes  \overline{  \xi_{1}  }   ) \\
& +         \epsilon_{x}    \prod_{l=0 }^{n-1} \langle   \overline{\xi_{m-l}} , \overline{ \eta_{n-l} }    \rangle_A H_\epsilon \left(  W( \overline{ \xi_{m-n}}  \otimes \ldots \otimes   \overline{  \xi_{1}  } )  \right). \\
\end{split}
\] 
By a similar argument we get also that,
\[
\begin{split}
H_{\epsilon}(H_{\epsilon}(xy^\ast)^\ast) = &
\sum_{k=0}^{n-1}     \epsilon_{x} \epsilon_{y}     \prod_{l=0 }^{k-1} \langle     \overline{\xi_{m-l}} , \overline{ \eta_{n-l} }  \rangle_A  W(   \eta_{1} \otimes \ldots \otimes    \eta_{n-k}  \otimes  \overline{  \xi_{m-k  }} \otimes \ldots \otimes  \overline{  \xi_{1}  }   ) \\
& +     \epsilon_{x}  \prod_{l=0 }^{n-1} \langle   \overline{\xi_{m-l} }  , \overline{ \eta_{n-l} }   \rangle_A H_\epsilon \left(  W( \overline{\xi_{m-n}}  \otimes \ldots \otimes   \overline{\xi_{1}}   )  \right). 
\end{split}
\]
Then,
\begin{equation}\label{Eqn=RHSCotlar}
\begin{split}
&  H_{\epsilon}( x H_{\epsilon}(y)^\ast) + H_{\epsilon}( y H_{\epsilon}(x)^\ast)^\ast -  H_{\epsilon}(H_{\epsilon}(xy^\ast)^\ast)^\ast \\
= &  \epsilon_{x} \epsilon_{y} \sum_{k=0}^{n}     \prod_{l=0 }^{k-1} \langle  \overline{ \eta_{n-l} } , \overline{  \xi_{m-l}}    \rangle_A  W( \xi_{1}  \otimes \ldots \otimes  \xi_{m-k}  \otimes  \overline{ \eta_{n-k} } \otimes \ldots  \otimes \overline{ \eta_{1} } ).
\end{split}
\end{equation}  
On the other hand, from \eqref{Eqn=xystar} we conclude that,
\begin{equation}\label{Eqn=LHSCotlar}
  H_{\epsilon}(x) H_{\epsilon}(y)^\ast  = \epsilon_{x} \epsilon_{y}  xy^\ast,
\end{equation}
which equals \eqref{Eqn=RHSCotlar} by \eqref{Eqn=Split}.

\vspace{0.3cm}

\noindent {\it Case 2:}  Assume $m = n >0$. Because as $H_\epsilon(1) = 0$ we find the following decomposition (so the summand $k= n$ vanishes),
\[
\begin{split}
H_{\epsilon}( x H_{\epsilon}(y)^\ast)
= & \sum_{k=0}^{n-1}   \epsilon_{y} \epsilon_{x}     \prod_{l=0 }^{k-1} \langle  \overline{ \eta_{n-l} }   , \overline{\xi_{m-l}}  \rangle_A  W( \xi_{1}  \otimes \ldots \otimes  \xi_{m-k}  \otimes  \overline{ \eta_{n-k} } \otimes \ldots \otimes \overline{ \eta_{1} } ).
\end{split}
\] 
Further, again as  as $H_\epsilon(1) = 0$, 
\[
\begin{split}
H_{\epsilon}( y H_{\epsilon}(x)^\ast)= &
\sum_{k=0}^{n-1}     \epsilon_{x} \epsilon_{y}     \prod_{l=0 }^{k-1} \langle  \overline{ \eta_{n-l} }   , \overline{\xi_{m-l} }  \rangle_A  W(   \eta_{1} \otimes \ldots \otimes    \eta_{n-k}  \otimes  \overline{  \xi_{m-k  } } \otimes \ldots \otimes  \overline{  \eta_{1}  }   ),  
\end{split}
\] 
and
\[
\begin{split}
H_{\epsilon}(H_{\epsilon}(xy^\ast)^\ast) = &
\sum_{k=0}^{n-1}     \epsilon_{x} \epsilon_{y}     \prod_{l=0 }^{k-1} \langle  \overline{ \eta_{n-l} }   , \overline{\xi_{m-l}}  \rangle_A  W(   \eta_{1} \otimes \ldots \otimes    \eta_{n-k}  \otimes  \overline{  \xi_{m-k  }} \otimes \ldots \otimes  \overline{  \xi_{1}  }   ).
\end{split}
\]
Then,  
\begin{equation}\label{Eqn=RHSCotlar2}
\begin{split}
&  H_{\epsilon}( x H_{\epsilon}(y)^\ast) + H_{\epsilon}( y H_{\epsilon}(x)^\ast)^\ast -  H_{\epsilon}(H_{\epsilon}(xy^\ast)^\ast)^\ast \\
= &  \epsilon_{x} \epsilon_{y} \sum_{k=0}^{n-1}     \prod_{l=0 }^{k-1} \langle  \overline{ \eta_{n-l} } , \overline{ \xi_{m-l}}    \rangle_A  W( \xi_{1}  \otimes \ldots \otimes  \xi_{m-k}  \otimes  \overline{ \eta_{n-k} } \otimes \ldots  \otimes \overline{ \eta_{1} } ).
\end{split}
\end{equation}
This expression is in the range of the projection  $\cE_\Omega^\perp$. On the other hand, by \eqref{Eqn=Split} and using $n = m$ we get
\begin{equation} 
\begin{split}
& \cE_\Omega^\perp \left( H_{\epsilon}(x) H_{\epsilon}(y)^\ast \right)  = \epsilon_{x} \epsilon_{y} \cE_\Omega^\perp \left( xy^\ast \right) \\
 = &
\epsilon_{x} \epsilon_{y} \sum_{k=0}^{n-1}     \prod_{l=0 }^{k-1} \langle  \overline{ \eta_{n-l} } , \overline{\xi_{m-l}}    \rangle_A  W( \xi_{1}  \otimes \ldots \otimes  \xi_{m-k}  \otimes  \overline{ \eta_{n-k} } \otimes \ldots  \otimes \overline{ \eta_{1} } ),
\end{split}
\end{equation}
which concludes the proof of Case 2 as this equals \eqref{Eqn=RHSCotlar2}. 
 

\vspace{0.3cm}

\noindent {\it Case 3:} Assume $n > m$. The proof can be obtained by a mutatis mutandis copy of Case 1. We sketch a second way to finish the proof. By Case 1:
\begin{equation}\label{Eqn=Case2}
H_{\epsilon}(y) H_{\epsilon}(x)^\ast = H_{\epsilon}( y H_{\epsilon}(x)^\ast) + H_{\epsilon}( x H_{\epsilon}(y)^\ast)^\ast -  H_{\epsilon}(H_{\epsilon}(yx^\ast)^\ast)^\ast.
\end{equation}
Then one verifies that
$
  H_{\epsilon}(H_{\epsilon}(yx^\ast)^\ast) =  H_{\epsilon}(H_{\epsilon}(xy^\ast)^\ast)^\ast$.
 So that taking adjoints of \eqref{Eqn=Case2} we see,
\[
H_{\epsilon}(x) H_{\epsilon}(y)^\ast =
 H_{\epsilon}( y H_{\epsilon}(x)^\ast)^\ast + H_{\epsilon}( x H_{\epsilon}(y)^\ast) - H_{\epsilon}(H_{\epsilon}(xy^\ast)^\ast)^\ast.
\]
\end{proof}

We write $D$ for the operator $D_{\varphi_{\Omega}}$.
 
\begin{lem} \label{Lem=OrthH}
    For $x \in \mathcal{A}$ we have that
	\[
	 \varphi_{\Omega}( H_\epsilon(  x)^\ast  H_\epsilon( x )  ) = \varphi_{\Omega}( x^\ast x ).
	\]
	So certainly for every $1 \leq p < \infty$ we get that  $\Vert D^{\frac{1}{2p}} \varphi_{\Omega}(  H_\epsilon(  x)^\ast  H_\epsilon( x ) )  D^{\frac{1}{2p}} \Vert_{p} = \Vert  D^{\frac{1}{2p}}  \varphi_{\Omega}( x^\ast x )  D^{\frac{1}{2p}} \Vert_{p}$.
\end{lem}
\begin{proof}
	As $x$ is in the algebra $\mathcal{A}$ we may take a decomposition $x = x^+ + x^-$ with $x^{\pm}$ in the linear span of Wick operators $W(\xi_1 \otimes \ldots \otimes \xi_n), \xi_1 \in \cHC^{\pm}$. We have
	\[
	\varphi_{\Omega}(H_\epsilon(x)^\ast H_\epsilon(x)) = \langle x^+ \Omega, x^+ \Omega \rangle + \langle x^- \Omega, x^- \Omega \rangle - \langle x^+ \Omega, x^- \Omega \rangle  -\langle x^- \Omega, x^+ \Omega \rangle. 
	\]
	As $\cHC^+$ and $\cHC^-$ are orthogonal for the inner product of $\cH$ we find that 
	\[
	\varphi_{\Omega}(H_\epsilon(x)^\ast H_\epsilon(x))  =  \langle x^+ \Omega, x^+ \Omega \rangle + \langle x^- \Omega, x^- \Omega \rangle = \varphi_{\Omega}(  x^\ast x  ).
	\] 
\end{proof} 

\begin{thm}
 For every $1<p<\infty$ and   every choice of  $\epsilon= (\cHC^+, \cHC^-)$ as in Definition \ref{Dfn=FreeHilbert} such that $A^{it}$ leaves $\cHC^{\pm}$ invariant for all $t \in \mathbb{R}$ the map $H_\epsilon$ extends to a bounded map $L_p(\cM) \rightarrow L_p(\cM)$ that is determined by
\begin{equation}\label{Eqn=LpHilbert}
H_\epsilon:  D^{\frac{1}{2p}} x  D^{\frac{1}{2p}} \mapsto  D^{\frac{1}{2p}} H_{\epsilon}(x)  D^{\frac{1}{2p}}, \qquad x \in \mathcal{A}.
\end{equation}
Moreover, let $c_p$ be the norm of \eqref{Eqn=LpHilbert}. Then for $p \geq 2$ a power of 2 we have $c_p \leq p^{\gamma/2}$ with  $\gamma = ^3\!\!\log(2)$. Further, for $C = 2^{\gamma/2}$ we have  $c_p \leq C p^{\gamma/2}$ for $p \geq 2$ arbitrary. 
\end{thm}
\begin{proof}
	For $p=2$ the map $H_\epsilon$ defines a contraction on $L_2(\cM)$ and so the statement is true.
	
The space $D^{\frac{1}{2p}} \mathcal{A} D^{\frac{1}{2p}}$ is dense in $L_p(\cM)$.   As $H_\epsilon: \mathcal{A} \rightarrow \mathcal{A}$ commutes with the modular automorphism group of $\varphi_{\Omega}$, c.f. Remark \ref{Rmk=GaussianModular}, it follows from a computation similar to \eqref{Eqn=KMSSym} that,
\begin{equation}\label{Eqn=HSigmaCom}
\begin{split}
& H_{\epsilon} ( D^{\frac{1}{p}} x) =
  D^{\frac{1}{p}}  H_{\epsilon}(x).
\end{split}
\end{equation}
We now estimate $c_{2p}$ in terms of $c_p$.
By Cotlar's identity \eqref{Eqn=Cotlar} and Lemma \ref{Lem=OrthH} we get that,
\[
\begin{split}
& \Vert  D^{\frac{1}{2p}} H_\epsilon(x) \Vert_{2p}^{2} \\
=  & \Vert D^{\frac{1}{2p}}  H_{\epsilon}(x)  H_{\epsilon}(x)^\ast  D^{\frac{1}{2p}} \Vert_{p} \\
\leq & 
\Vert D^{\frac{1}{2p}} \mathcal{E}_{\Omega}^\perp \left( H_{\epsilon}(x)  H_{\epsilon}(x)^\ast \right)  D^{\frac{1}{2p}} \Vert_{p} + \Vert D^{\frac{1}{2p}} \varphi_{\Omega}\left( H_{\epsilon}(x)  H_{\epsilon}(x)^\ast \right)  D^{\frac{1}{2p}} \Vert_p  
 \\
\leq & \Vert   D^{\frac{1}{2p}} H_{\epsilon}( x H_{\epsilon}(x)^\ast)  D^{\frac{1}{2p}} \Vert_p + \Vert  D^{\frac{1}{2p}} H_{\epsilon}( x H_{\epsilon}(x)^\ast)^\ast  D^{\frac{1}{2p}} \Vert_p \\
& \qquad \qquad  +  \Vert  D^{\frac{1}{2p}} H_{\epsilon}(H_{\epsilon}(xx^\ast)^\ast)^\ast  D^{\frac{1}{2p}} \Vert_p + 
\Vert   D^{\frac{1}{2p}} \varphi_{\Omega}\left( x x^\ast \right)  D^{\frac{1}{2p}} \Vert_{2p}
\\
\leq & c_p \Vert  D^{\frac{1}{2p}} x \Vert_{2p} \Vert  H_{\epsilon}(x)^\ast  D^{\frac{1}{2p}} \Vert_{2p} +
 c_p \Vert  D^{\frac{1}{2p}} x \Vert_{2p} \Vert  H_{\epsilon}(x)^\ast  D^{\frac{1}{2p}} \Vert_{2p} \\
 & \qquad + c_p^2 \Vert  D^{\frac{1}{2p}} x\Vert_{2p}^2  + \Vert  D^{\frac{1}{2p}} x \Vert_{2p}^2 \\
= &  2 c_p \Vert  D^{\frac{1}{2p}} x \Vert_{2p} \Vert  H_{\epsilon}(x)^\ast  D^{\frac{1}{2p}} \Vert_{2p}  + (c_p^2+1) \Vert  D^{\frac{1}{2p}} x \Vert_{2p}^2. 
 \end{split}
\]
By density we conclude that $c_{2p} \leq  c_p + \sqrt{2 c_p^2 +1}$. In particular $c_{2p} \leq  (1+ \sqrt{2} ) c_p$ from which it follows that for $p$ a power of 2 we get that $c_p \leq p^{\gamma}$ with $\gamma = \frac{\log(2)}{\log(1 + \sqrt{3})}$.  For other $p \geq 2$ the result follows by interpolation, see \cite{BerghLofstrom}, \cite{TerpII}.
\end{proof}

\begin{rmk}
 We do not know what the optimal constants are for the norm of  $H_\epsilon$ on $L_p(\cM)$. We also leave it as an open question whether the Hilbert transform is a bounded map $L_\infty \rightarrow \BMOa$ or even  $\BMOa \rightarrow \BMOa$ as for the classical Hilbert transform. 
\end{rmk}

\subsection{Khintchine type BMO inequalities and multipliers}
We provide examples of $L_\infty \rightarrow \BMOa$-multipliers on free Araki-Woods factors. Earlier results on non-commutative  $L_\infty \rightarrow \BMOa$-multipliers in the tracial setting were obtained by Mei \cite{MeiShortNote} but here we do not need to appeal to lacunary sets. We  use  the Markov semi-group $\cS = (\Psi_{t} = \Phi_{e^{-t}})_{t \geq 0}$ that is determined by
\begin{equation}\label{Eqn=PhiGroup}
\Phi_{r}: W(\xi_1 \otimes \ldots \otimes \xi_n) \mapsto r^n W(\xi_1 \otimes \ldots \otimes \xi_n), \qquad 0 \leq r \leq 1. 
\end{equation}
This semi-group is well-known to be Markov, KMS-symmetric and $\varphi_\Omega$-modular. 

\begin{prop} \label{Prop=BMOKhintchine}
	Suppose that $\cH$ is infinite dimensional. Let $(e_k)_k$ be a set of vectors in $\cHC$ that are orthogonal in $\cH$. 
For $i = (i_1, \ldots, i_n)$ a multi-index  set $e_{i} = e_{i_1} \otimes \ldots \otimes e_{i_n}$ and $\overline{e}_{i} = \overline{e}_{i_1} \otimes \ldots \otimes \overline{e}_{i_n}$. Take $F$ a set of multi-indices such that $\langle e_{i_1}, e_{j_1} \rangle_A = \langle \overline{e}_{i_1}, \overline{e}_{j_1} \rangle_A   = 0$   if $i \not = j$.  We have, for any $x = \sum_{i \in F} c_i W(e_i)$ with $c_i \in \mathbb{C}$  a finite sum of Wick operators whose frequency support lies in  $F$,  that,
\[
\begin{split}
& \Vert x \Vert_{\BMO_{\mathcal{S}}}^2 \leq 2 \max \left\{ \Vert \sum_{i} \vert c_i \vert^2 W(e_i)^\ast W(e_i) \Vert,  \Vert \sum_{i} \vert c_i \vert^2 W(e_i) W(e_i)^\ast \Vert  \right\}.
\end{split}
\]
 
 \end{prop}
\begin{proof}
Using the definition of $x$ and the triangle inequality,
\begin{equation}\label{Eqn=BMOAraki}
	\begin{split}
 &  \Vert \Phi_r(x)^\ast \Phi_r(x) - \Phi_r(x^\ast x) \Vert  \\
 \leq & \Vert \sum_{i =j}   \overline{c_i} c_j (r^{\vert i \vert + \vert j \vert}   -   \Phi_r )( W(e_i)^\ast W(e_j)   ) \Vert + \Vert \sum_{ i \not = j}   \overline{c_i} c_j (r^{\vert i \vert + \vert j \vert}   -   \Phi_r )( W(e_i)^\ast W(e_j)   ) \Vert 
 \end{split}
 \end{equation}
 If $i \not = j$ we get that $\langle e_{i_{1}}, e_{j_{1}} \rangle_A = 0$, so that by \eqref{Eqn=Split} we see that  $ W( e_i)^\ast W( e_j) = W(e_{i}^\ast \otimes e_j)$ where $e_i^\ast = \overline{e}_{i_n} \otimes \ldots 
 \otimes  \overline{e}_{i_1}$ so that, 
 \[
 (r^{\vert i \vert + \vert j \vert}   -   \Phi_r )( W( e_i)^\ast W( e_j)   ) =
 (r^{\vert i \vert + \vert j \vert}   -   r^{\vert i \vert + \vert j \vert} )(   W( e_{i}^\ast  \otimes  e_j) ) = 0.
 \]
 Therefore we continue \eqref{Eqn=BMOAraki} by  using that $\Phi_r$ is a ucp map and that  the expression $\sum_{i}    \vert c_i \vert^2      W( e_i)^\ast W( e_i)$ is a summation of positive elements,
 \[
 \begin{split}
 \Vert \Phi_r(x)^\ast \Phi_r(x) - \Phi_r(x^\ast x) \Vert  = &  \Vert \sum_{i}    \vert c_i \vert^2 (r^{2 \vert i \vert}   -   \Phi_r )( W( e_i)^\ast W( e_i)   ) \Vert \\  
  \leq & 2 \Vert \sum_{i }   \vert c_i \vert^2 ( W( e_i)^\ast W( e_i)   ) \Vert.
  \end{split}
	\]
 This shows that 
 \[
 \Vert x \Vert_{\BMO^c}^2 \leq 2 \Vert \sum_{i} \vert c_i \vert^2 W(e_i)^\ast W(e_i) \Vert.
 \]
 Then in the same way using the orthogonality $\langle \overline{e}_{i_1}, \overline{e}_{j_1} \rangle = 0, i \not = j$ we get that  $\Vert x \Vert_{\BMO^r}^2  = \Vert x^\ast \Vert_{\BMO^c}^2  \leq 2 \Vert \sum_{i} \vert c_i \vert^2 W(e_i) W(e_i)^\ast \Vert$.
\end{proof}

Let $\delta_F$ be the indicator function on a set $F$.  The following Corollary \ref{Cor=Int} is a consequence of our interpolation result of Theorem \ref{Thm=InterpolationBMO}. We call $n$ the length of a multi-index $i = (i_1, \ldots, i_n)$.

\begin{cor}\label{Cor=Int}
	Take $F$ a set of multi-indices of length at most $n$ such that $\langle e_{i_1}, e_{j_1} \rangle_A = \langle \overline{e}_{i_1}, \overline{e}_{j_1} \rangle_A   = 0$   if $i \not = j$.  The projection $P_F: W(e_i) \mapsto \delta_F(i) W(e_i)$
	extends to a bounded map $\cM \rightarrow \BMO_{\cS}(\cM)$. Consequently, $P_F$ determines a bounded map $P_F^{(p)}: L_p(\cM) \rightarrow L_p(\cM)$ given by $P_F^{(p)}:  D^{\frac{1}{2p}} x D^{\frac{1}{2p}} \mapsto  D^{\frac{1}{2p}}P_F( x) D^{\frac{1}{2p}}$. 
\end{cor}
\begin{proof}
	Let $x = \sum_{i} c_i W(e_i) \in \mathcal{A}$. 
	 As for any $i \in F$ its length is bounded by $n$  we have that $\Vert W(e_i)^\ast W(e_i) \Vert \leq C$ for some constant $C$. We get that
	 $	  \sum_{i \in F} \vert c_i \vert^2 \Vert  W(e_i)^\ast W(e_i) \Vert \leq C
	 \sum_{i \in F} \vert c_i \vert^2  = \Vert x \Vert_2^2$. 
	 	 Proposition \ref{Prop=BMOKhintchine} shows therefore that we get the first inequality in
	\[
	  \Vert  P_F(x)  \Vert_{\BMO_{\mathcal{S}}} \leq \sqrt{2} \Vert P_F(x) \Vert_2 \leq \sqrt{2} \Vert x \Vert_2 \leq \sqrt{2} \Vert x \Vert_\infty.
	\] 
In Section \ref{Sect=Dilation} we show that   $(\Psi_t)_{t\geq 0}$ has a Markov dilation with a.u. continuous path.	We then get $L_p \rightarrow L_p$ boundedness of $P_F$ by interpolation, see Theorem \ref{Thm=InterpolationBMO}.   
\end{proof}

\subsection{A Markov dilation for the radial semi-group of free Araki-Woods factors}\label{Sect=Dilation}
In this Section we show that the radial semi-group on free Araki-Woods factors has a good reversed Markov dilation. 
The first step in the proof of  Proposition \ref{Prop=MarkovDilate} is  due to Ricard (see the final remarks of \cite{RicardDilation}). We need to find a suitable analogue for semi-groups which we do by an ultraproduct argument. Similar techniques were used in  \cite{Arhancet}, \cite{ArhancetEtAl}  though in this case through quantization we can give a shorter argument directly on the Hilbert space level, see also the comment below this proposition.

\begin{prop}\label{Prop=MarkovDilate}
 For $t \geq 0$ consider the Markov semi-group  $\Psi_t = \Phi_{e^{-t}}$ where $\Phi_r, 0 < r \leq 1$ is the Markov map on $\cM$ determined by \eqref{Eqn=PhiGroup}.
  $\Psi_t$ admits a $\varphi_{\Omega}$-modular Markov dilation with a.u. continuous path. 
 \end{prop}
\begin{proof} For $t \geq 0$. Set $T_t \in B(\cH)$ by $T_t \xi = e^{-t} \xi$.   The proof splits in steps. 
	
	\vspace{0.3cm}
	
	 \noindent {\it Step 1: Constructing a dilation for subsemi-groups of $(T_t)_{ t \geq 0}$.} 
	 	 	Firstly for each $t \geq 0$ we may find a Hilbert space $\cK$ containing $\cH$ with orthogonal projection $P_{\cH}: \cK_t \rightarrow \cH$ and a unitary $U_t \in B(\cK)$ such that  for every $l \in \mathbb{N}$,
	\begin{equation}\label{Eqn=DilateU}
	P_{\cH} U_t^l \vert_{\cH} =  T_t^l = T_{tl}.
	\end{equation}
	Indeed, the Hilbert space $\cK = \ell_2(\mathbb{Z}) \otimes \cH$ with unitary  $U_t, t = - \log(r)$ acting on the first tensor leg by 
	\[
	\left( 	
	\begin{array}{cccccccccccccc}
\ddots & \vdots	&   \vdots     &       &  &&&&  &&&&&\\
\ldots &  1	&  0      & 0       & \ldots &&&&  &&&&&\\
\ldots &  0  & 1      & 0      & \ldots  &&&& &&&&&\\
\ldots &  0  & 0      & 1      & \ldots&&&  &&&&&&\\
	   & \vdots & \vdots & \vdots & \ddots&&&& &&&&&\\
	   & 	  	&        &	      & \ldots   & 1    & 0            &  0  &       0 &\ldots &&&&\\
	 	&&&	&      				\ldots   & 0    & \sqrt{1-r^2}             & r &0& \ldots&&&& \\
	    &&&  & 					\ldots   & 0	& r&   \sqrt{1-r^2}        & 0&    \ldots&&&	& \\
	       & 	  	&        &	      & \ldots   & 0    & 0            &  0  &       1 &\ldots &&&&\\
&&&&&&&&& \ddots & \vdots	&   \vdots     &       &  \\
&&&&&&&&&\ldots &  1	&  0      & 0       & \ldots  \\
&&&&&&&&&\ldots &  0  & 1      & 0      & \ldots   \\
&&&&&&&&&\ldots &  0  & 0      & 1      & \ldots \\
&&&&&&&&& & \vdots & \vdots & \vdots & \ddots 
	\end{array}
	\right),
	\]
	where the bottom left entry  $r$ is located at position $(0,0)$. 
	So $U_t$ acts as a shift operator on $\ell_2(\mathbb{Z} \backslash \{ 0,1 \} ) \otimes \cH$.  $\cH$ is a subspace of $\cK$ by the embedding 
	\[
	J: \xi \mapsto \delta_0 \otimes \xi \in  \ell_2(\mathbb{Z}) \otimes \cH = \ell_2(\mathbb{Z}, \cH).
	\]	
	\eqref{Eqn=DilateU} is then elementary to check (see also \cite[Theorem 1.1]{PisierBook}). We let $P_{t,n}$ be the orthogonal projection onto the closed linear span of 
	$\left\{ U_{t}^l  \xi \mid   l \geq n, \xi \in \cH  \right\}$.
	We get that for $\xi, \eta \in \cH$ we have for $l,n \geq 0$, 
	\begin{equation}\label{Eqn=UComp}
	\langle \xi, U_t^{n + l} \eta \rangle 
	= \langle \xi, P_{\cH} U_t^{n+l} P_{\cH} \eta \rangle 
	= \langle  T_t^{n+l} \xi,  \eta \rangle 
	= \langle  T_t^{n} \xi,  U_t^{l} \eta \rangle 
	= \langle U_t^n T_t^{n} \xi,  U_t^{n+l} \eta \rangle.
	\end{equation}
	Which shows that $P_{t,n} P_{\cH} \xi = U_t^n T_t^n \xi$. Moreover from \eqref{Eqn=UComp} we   get for $n \geq k$  that 
		$\langle U_{t}^k \xi, U_t^{n + l} \eta \rangle  = \langle U_{t}^{n} T_t^{n-k} \xi,  U_t^{n+l} \eta \rangle$.
	So we find that $P_{t,n} U_{t}^{k} \xi = U_t^n T_t^{n-k} \xi$. So if we put $J_{t,n} =  U_{t}^{n} J: \cH \rightarrow \cK$ we get that 
	 \begin{equation}\label{Eqn=DiscrDilation}
	   P_{t,n} J_{t,k} = J_{t,n} T^{n-k}_t, \qquad n \leq k.
	 \end{equation}
	 This is a discrete Hilbert space version of the reversed Markov dilation property \eqref{Eqn=RevMarkovDilationRelation}.

	 \vspace{0.3cm}

	 \noindent {\it Step 2: Constructing a Markov dilation.} We shall now construct a continuous version of \eqref{Eqn=DiscrDilation}. To do so, for $t \geq 0$ let $\cK_t := \cK$ be the Hilbert space as in the previous paragraph and let $J_{t,n}: \cH \rightarrow \cK_t$ be the injection as before. Also let $P_{t,n}$ and $U_t$ be  as before.

	Set groups $\sfG_m = 2^{-m} \mathbb{Z}$ and $\sfG = \cup_{m \geq 1} \sfG_m$. The group $\sfG$ is understood as a topological group with the Euclidian topology inherited from $\mathbb{R}$. 	
	Let $\cU$ be a non-principal ultrafilter on $\mathbb{N}$. Consider $\cK_\cU = \prod_{m, \cU} \cK_{2^{-m}}$. Let $K: \cH \rightarrow \cK_\cU$ be the embedding sending $\xi$ to the constant family $(\xi)_{\cU}$. Let $P_{\cH} = K^\ast$ be the projection onto $\cH$.   
		For $t \in \sfG$ we define the unitary $V_{t,m}$ on $\cK_{ 2^{-m} }$ by
	\[
	 V_{t,m}   =
	 \left\{
	 \begin{array}{ll}
	 	 U_{ 2^{-m} }^{t 2^m}     &  {\rm  if } \:   t \in \sfG_m, \\
	 	  \Id_{\cK_{ 2^{-m}}}     &  {\rm otherwise.}
	 \end{array}
	 \right.
	\]
	Then for $t \in \sfG$ set $V_t =( V_{t,m})_{\cU}$ which is a unitary on $\cK_{\cU}$. We claim that the assignment 
	\begin{equation}\label{Eqn=sfGEmbedding}
	\sfG \ni t \mapsto  V_t P_{\cK}
	\end{equation}
	 is strong-$\ast$ continuous. Indeed, let $\xi \in \cH$ be a unit vector. Let $t,s \in \sfG$ and assume that $s \geq t \geq 0$. Then let $M$ be such that for    any  $m > M$ we have $s, t \in \sfG_m$. Fix such $m > M$. We get that
	 \[
	 \begin{split}
	   U_{2^{-m}}^{t 2^m}     \xi = &	 
	  (0, \ldots, 0, \sqrt{1-r^2}   \xi, \sqrt{1-r^2} r  \xi, \ldots, \\
	  & \ldots,  \sqrt{1-r^2} r^{t2^m -3} \xi, \sqrt{1-r^2} r^{t2^m -2} \xi, \sqrt{1-r^2} r^{t2^m -1} \xi, r^{t2^m} \xi , 0 , 0 , \ldots  ). \\
	  & 
	  \end{split}
	 \]
	 Recalling $r = e^{-2^{-m}}$ this shows that we get from a small elementary computation,
	 \begin{equation}\label{Eqn=StrongComputation}
	 \begin{split}
	& \Vert ( U_{2^{-m}}^{s 2^m}  -  U_{2^{-m}}^{t 2^m}  )   \xi  \Vert^2_2 \\
	 = & ( r^{s 2^m} - r^{t 2^m})^2
	 + (1 - r^2) \sum_{l=1}^{t 2^m} (   r^{s 2^m - l } -   r^{t 2^m - l } )^2 
	  + (1 - r^2) \sum_{l= t 2^m+ 1}^{s 2^m} (r^{s2^m - l})^2 \\
	= &  (e^{-2s}  - e^{-2t})^2 + (e^{-2s} - e^{-2r})^2 (e^{-2t} -1) + e^{-2(s-t)} (e^{-2(s-t)} - 1),
	 \end{split}
	 \end{equation}
	 which converges to 0 as $s \rightarrow t$. This shows that the unitary group $t \mapsto V_t P_{\cH}$ is strong-$\ast$ continuous.    
	We extend \eqref{Eqn=sfGEmbedding} to a strongly continuous map  $\mathbb{R} \ni t \rightarrow  V_t P_{\cH}$. 	This shows that we get an isometric embedding for every $t \in \mathbb{R}$, 
	\[
	   K_t: \cH \rightarrow \cK_{\cU}:  \xi \mapsto  V_{t}  K \xi.
	\] 
	For $t \in \sfG$  and $m \in \mathbb{N}_{\geq 1}$ we define
	\[
	  Q_{ t,m} =  
	  \left\{
	  \begin{array}{ll}
	  	  P_{2^{-m}, s 2^m}   & {\rm if } \:  t \in \sfG_m,    \\
	  	  0 				  & {\rm otherwise}.   
	  \end{array}	  
	  \right.
	\]
	Then set $Q_{t} = (Q_{ t,m})_{m, \cU}$. We claim that the mapping $\sfG \ni t \mapsto Q_t$ is decreasing and strongly continuous. Indeed we have for $t \in \sfG_m$ that $P_{2^{-m}, t 2^m} = P_{2^{-m}, 0} U_{2^{-m}}^{- t 2^m}$. Set $P = (P_{2^{-m}, 0})_{\omega}$. So that for $t \in \sfG$ we have $Q_t = P  V_t^{\ast}$. A computation similar to \eqref{Eqn=KMSSym} shows that the function $\sfG \ni t \mapsto P V_t^\ast$ is weakly continuous. But as $Q_t$ is decreasing this convergence actually holds in the strong topology (see \cite[Theorem 4.1.1]{Murphy}) and by self-adjointness in the strong-$\ast$-topology.
		Therefore we obtain a decreasing strong-$\ast$ continuous   map $\mathbb{R} \ni t \mapsto Q_t$.  
	
	 For $s,t \in \sfG, s \geq t$ and any $m$ large such that $s,t \in \sfG_m$. We get that for $\xi \in \cH$, 
	\[
	    Q_{s,m}  V_{t,m} \xi =  P_{2^{-m}, s 2^m} U_{ 2^{-m} }^{t 2^m} \xi = U_{ 2^{-m} }^{s 2^m}  T_{2^{-m}}^{(s-t)2^m} \xi = V_{s,m}    T_{s-t} \xi.
	\]
	This shows that for $s,t \in \sfG, s\geq t$ we get that 
	$Q_s V_t J = V_s J T_{s-t}$. 
	 By strong continuity we get $Q_s V_t J = V_s J T_{s-t} $ for all $s \geq t \geq 0$. So by definition 
	 \begin{equation}\label{Eqn=HilbertRel}
	 Q_s J_t  = J_s T_{s-t} \qquad \textrm{ for all } s \geq t \geq 0.
	 \end{equation}
	  We finish the proof by quantization. 	Let $(V_t^\cK)_{t \in \mathbb{R}} = (\Id_{B(\ell_2(\mathbb{Z}))} \otimes V_t)_{t \in \mathbb{R}}$ be the orthogonal transformation group on $\cK_{\mathbb{R}} = \ell_2(\mathbb{Z}) \otimes \cH_{\mathbb{R}}$.  
	We set $\cN = \Gamma(\cK, (V_t^\cK)_{t \in \mathbb{R}}  )$ and $\cN_s = \Gamma( Q_s \cK), s \geq 0$. 
	 By second quantization we get a conditional expectation $\cE_s := \Gamma(Q_{s}): \cN \rightarrow \cN_s$ and a normal injective $\ast$-homomorphism $\pi_s = \Gamma(J_s)$. By \eqref{Eqn=HilbertRel} they satisfy
	 \[
	   \cE_s \circ \pi_t = \pi_s \circ \Psi_{s-t}, \qquad 0 \leq t \leq s.  
	 \]
	 It is clear from Remark \ref{Rmk=GaussianModular} that this dilation is modular. 
	 
	 \vspace{0.3cm}

\noindent {\it Step 3: A.u. continuity.} Suppose that $\xi_i$ is a net of vectors in $\cH$ converging in norm  to $\xi \in \cH$. Then we have for creation operators $a(\xi_t) \rightarrow a(\xi)$ in norm. By the Wick Theorem \ref{Thm=Wick} we see that $W(\xi_t) \rightarrow W(\xi)$ in norm. Now for $x = W(\xi) \in \mathcal{A}, \xi \in \cH$     the martingale   $m_t(x) = \cE_t (\pi_t(x)) = W( Q_t J_t \xi)$ is norm continuous.  

\end{proof}

 \begin{rmk}\label{Rmk=Nagy}
 Proposition \ref{Prop=MarkovDilate} could potentially  also be derived from a suitable analogue of \cite[Theorem 7.1]{Nagy}, provided that in this theorem one can keep track of the location of a specified real Hilbert subspace.  
\end{rmk}

\subsection*{Acknowledgements}
The author thanks A. Gonz\'alez-Perez and M. Junge for useful disucssions on BMO-multipliers.  The author thanks M. Veraar for pointing out   \cite{Nagy} and Remark \ref{Rmk=Nagy}.  The author thanks the referee for useful remarks leading to an improvement of the manuscript.

\end{document}